\documentclass[english,12pt]{amsart}

\usepackage[latin9]{inputenc}	
\usepackage{babel}
\usepackage{geometry,graphicx}  
\usepackage{amsthm}
\usepackage{amsmath}	
\usepackage{amssymb}
 \usepackage{amsfonts}
 \usepackage{amscd}
\usepackage[dvipsnames]{xcolor}
\usepackage{pdfsync}
\usepackage{url}

\usepackage{amsthm}

\makeatletter
\theoremstyle{remark}
\newtheorem*{rem*}{\protect\remarkname}
\makeatother
\usepackage{babel}
\providecommand{\remarkname}{Remark}

\theoremstyle{plain}

 \newtheorem{theorem}{\protect\theoremname}[section]
 \newtheorem{corollary}[theorem]{\protect\corollaryname}
 \newtheorem{proposition}[theorem]{\protect\propositionname}
 
 \newtheorem{algorithm}[theorem]{Algorithm}
\theoremstyle{definition}
  \newtheorem{example}[theorem]{\protect\examplename}
  
\theoremstyle{remark}
  \newtheorem{remark}[theorem]{\protect\remarkname}

\makeatother

\usepackage{babel}
\providecommand{\definitionname}{Definition}
\providecommand{\examplename}{Example}
\providecommand{\theoremname}{Theorem}
\providecommand{\corollaryname}{Corollary}
\providecommand{\propositionname}{Proposition}
\providecommand{\lemmaname}{Lemma}

\newcounter{FNC}[page]
\def\fauxfootnote#1{{\addtocounter{FNC}{2}${\color{magenta}^\fnsymbol{FNC}}$%
     \let\thefootnote\relax\footnotetext{{\color{magenta}$^\fnsymbol{FNC}$#1}}}}
\def\defcolor#1{{\color{blue}{#1}}}
\def\demph#1{\defcolor{{\sl #1}}}
\newcommand{\bC}{\mathbb C}
\newcommand{\bG}{\mathbb G}
\newcommand{\bK}{\mathbb K}
\newcommand{\bL}{\mathbb L}
\newcommand{\bP}{\mathbb P}
\newcommand{\bQ}{\mathbb Q}
\newcommand{\bR}{\mathbb R}

\newcommand{\CP}{{\it CP}}

\newcommand{\bv}{{\bf v}}
\newcommand{\bw}{{\bf w}}

\newcommand{\cD}{\mathcal D}
\newcommand{\cG}{\mathcal G}
\newcommand{\cM}{\mathcal M}
\newcommand{\cL}{\mathcal L}
\newcommand{\cO}{\mathcal O}
\newcommand{\cV}{\mathcal V}

\DeclareMathOperator{\Wr}{Wr}

\begin{document}

\title{Numerical computation of Galois groups}

\author[J.~D.~Hauenstein]{Jonathan D.~Hauenstein}
\address{Department of Applied \& Computational Mathematics \& Statistics\\
         University of Notre Dame\\
         Notre Dame, IN  46556\\         
         USA}
\email{hauenstein@nd.edu}
\urladdr{\url{http://www.nd.edu/\~jhauenst}}

\author[J.~I.~Rodriguez]{Jose Israel Rodriguez}
\address{Department of Statistics\\
         University of Chicago\\
         Chicago, IL 60637\\         
         USA}
\email{JoIsRo@uchicago.edu}
\urladdr{\url{http://home.uchicago.edu/\~joisro}}

\author[F.~Sottile]{Frank Sottile}
\address{Department of Mathematics\\
         Texas A\&M University\\
         College Station, TX 77843\\
         USA}
\email{sottile@math.tamu.edu}
\urladdr{\url{http://www.math.tamu.edu/\~sottile/}}

\thanks{Research of Hauenstein supported in part by NSF grant ACI-1460032, Sloan Research
         Fellowship, and Army Young Investigator Program (YIP)} 
\thanks{Research of Rodriguez supported in part by NSF grant DMS-1402545}         
\thanks{Research of Sottile supported in part by NSF grant DMS-1501370}

\begin{abstract}
 The Galois/monodromy group of a family of geometric problems or equations is a subtle invariant that encodes 
 the structure of the solutions.
 Computing monodromy permutations using numerical algebraic geometry gives information about the group, but can
 only determine it when it is the full symmetric group.
 We give numerical methods to compute the Galois group and study it when it is not the full
 symmetric group.
 One algorithm computes generators while the other gives information on its structure as a permutation group.
 We illustrate these algorithms with examples using a {\tt Macaulay2} package we are developing
 that relies upon {\tt Bertini} to perform monodromy computations.
\end{abstract}
\maketitle

\section{Introduction}
Galois groups, which are a pillar of number theory and arithmetic geometry, encode the structure of field
extensions.
For example, the Galois group of the cyclotomic extension of~$\bQ$ given by the polynomial
$x^4+x^3+x^2+x+1$ is the cyclic group of order four, and not the full symmetric group.  
A finite extension $\bL/\bK$, where $\bK$ has transcendence degree~$n$ over~$\bC$, corresponds to a branched
cover $f\colon V\to U$ of complex algebraic varieties of dimension~$n$, with~$\bL$ the function field of $V$
and $\bK$ the function field of $U$.
The Galois group of the Galois closure of $\bL/\bK$ equals the monodromy group of the branched
cover~\cite{Harris79,Hermite}. 
When~$U$ is \mbox{rational, $f\colon V\to U$}~may be realized as a family of polynomial systems rationally parameterized
by points of $U$.   
Applications of algebraic geometry and enumerative geometry are sources of such families.
For these, internal structure such as numbers of real solutions and symmetry of the original problem
are encoded in the Galois/monodromy group.

Computing monodromy is a fundamental operation in numerical algebraic geometry.
Computing monodromy permutations along randomly chosen loops in the base $U$ was used in~\cite{LS09}
to show that several Schubert problems had Galois/monodromy group the full symmetric group.
Leaving aside the defect of that computation---the continuation (and hence the monodromy permutations) was not
certified---this method only computes an increasing sequence of subgroups of the Galois group, and thus only determines the Galois group when it is the full symmetric group.
In all other cases, this method lacks a stopping criterion.

We offer two additional numerical methods to obtain certifiable information about Galois groups and
investigate their efficacy.
The first method is easiest to describe when $U$ is a rational curve so that $\bK=\bC(t)$, the field of
rational functions. 
Then $V$ is an algebraic curve $C$ equipped with a dominant map $f\colon C\to\bC$ whose fiber at $t\in\bC$
consists of solutions to a polynomial system that depends upon $t$.
This is a degree $k$ cover outside the branch locus $B$, which is a finite subset of $\bC$.
The monodromy group of $f\colon C\to \bC$ is generated by permutations coming from loops encircling each
branch point.

Our second method uses numerical irreducible decomposition of the $s$-fold fiber product to determine 
orbits of the monodromy group acting on $s$-tuples of distinct points in a fiber.
When $s=k{-}1$, this computes the Galois group.
The partial information obtained when $s<k{-}1$ may be sufficient to determine the Galois group.

We illustrate these methods.
The irreducible polynomial $x^4-4x^2+t$ over $\bC(t)$ defines a curve \defcolor{$C$} in 
$\bC_x\times\bC_t$ whose projection $C\to\bC_t$ is four-to-one for $t\not\in B=\{0,4\}$.
The fiber above the point $t=3$ is $\{-\sqrt{3},-1,1,\sqrt{3}\}$.
%
%
%
Following these points along a loop in $\bC_t$ based at $t=3$ that encircles the branch
point $t=0$ gives the $2$-cycle $(-1,1)$. 
A loop encircling the branch point $t=4$ gives the product of
$2$-cycles, $(-\sqrt{3},-1)(1,\sqrt{3})$.
These permutations generate the Galois group, which is isomorphic to the dihedral group $D_4$ and has order 8. 
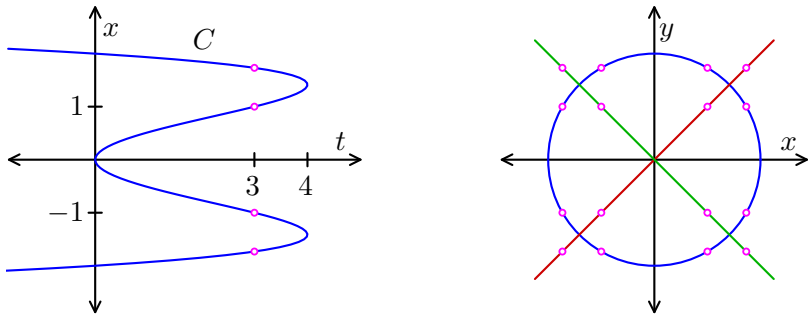
\begin{figure}[htb]
  \begin{picture}(135.5,117)(0.5,0.5)
    \put(-0.5,-0.5){\includegraphics{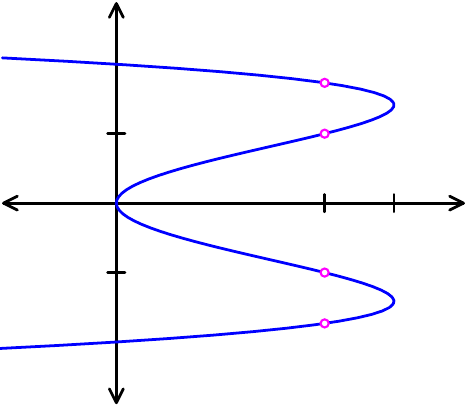}}
    \put(90,45){\small$3$}\put(110,45){\small$4$}
    \put(124,62){\small$t$} \put(36,105){\small$x$} 
    \put(15,35){\small$-1$} \put(70,100){\small$C$} 
    \put(23.5,75){\small$1$}
  \end{picture}
   \qquad \qquad
  \begin{picture}(117,117)(0.5,0.5)
    \put(-0.5,-0.5){\includegraphics{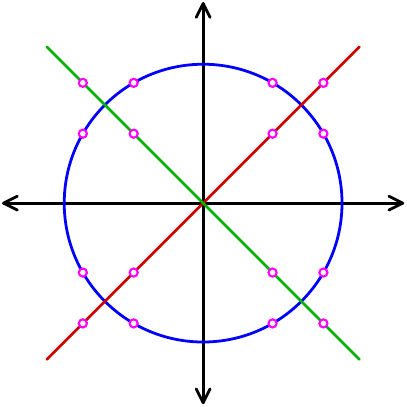}}
    \put(106,62){\small$x$} \put(60,105){\small$y$} 
  \end{picture}
 \caption{Curve $C$ over $\bC_t$ and fiber of $C\times_{\bC_t}C$ over $t=3$.}
 \label{F:first_example}
\end{figure}

The fiber product $C\times_{\bC_t}C$ consists of triples $(x,y,t)$, where $x$ and $y$ lie in the
fiber of $C$ above $t$.
It is defined in $\bC_x\times\bC_y\times\bC_t$ by the polynomials $x^4-4x^2+t$ and $y^4-4y^2+t$.
Since
\[
  (x^4-4x^2+t)-(y^4-4y^2+t)\ =\ 
  (x-y)(x+y)(x^2+y^2-4)\,,
\]
it has three components.
One is the diagonal defined by $x-y$ and $x^4-4x^2+t$.
The off-diagonal consists of two irreducible components, which implies that the action of the Galois group $\cG$ is
not two-transitive. 
One component is defined by $x+y$ and $x^4-4x^2+t$.
Its fiber over $t=3$ consists of the four ordered pairs $(\pm\sqrt{3},\mp\sqrt{3})$ and $(\pm1,\mp1)$, which is 
an orbit of $\cG$ acting on ordered pairs of solutions.
This implies that $\cG$ acts imprimitively as it fixes the partition $\{-\sqrt{3},\sqrt{3}\}\sqcup\{-1,1\}$.
Thus $\cG\subset S_4$ contains no $3$-cycle, so $\cG\subset D_4$.
The third component is defined by $x^2+y^2-4$ and $x^4-4x^2+t$ and its projection to $\bC_t$ has degree eight.
Thus $\cG$ has an orbit of cardinality eight, which implies 
$|\cG|\geq 8$, from which we can conclude that~$\cG$ is indeed the dihedral group $D_4$.\smallskip

The systematic study of Galois groups of families of geometric problems and equations coming from
applications is in its infancy.
Nearly every case we know where the Galois group has been determined exhibits a striking dichotomy
(e.g.,~\cite{BdCS,Harris79,J1870,LS09,MSJ,MNMH,RSSS,SW_double,Va}): either the group acts imprimitively, so
that it fails to be 2-transitive, or it is at least $(k{-}2)$-transitive 
in that it contains the alternating group (but is expected to be the full symmetric group).
The methods we develop here are being used~\cite{MNMH} to further investigate Galois
groups and we expect they will help to develop Galois groups as a tool to study geometric
problems, including those that arise in applications. 

The paper is structured as follows. 
Section~\ref{S:background} introduces the background material including permutation groups, Galois groups, fundamental groups, fiber products, homotopy continuation, and witness sets. 
In Section~\ref{S:Branch}, we discuss the method of computing monodromy by determining the branch locus, 
illustrating this on the classical problem of determining the monodromy group
of the 27 lines on a cubic surface.
In Section~\ref{S:fiber}, we discuss using fiber products to obtain information about the Galois group,
illustrating this method with the monodromy action on the lines on a cubic surface. 
We further illustrate these methods using three examples from applications in Section~\ref{S:examples}, and we
give concluding remarks in Section~\ref{S:conclusion}.

\section{Galois groups and numerical algebraic geometry}\label{S:background}

We describe some background, including permutation groups, Galois/monodromy groups, and
fundamental groups of hypersurface complements from classical algebraic geometry, as well as the topics from
numerical algebraic geometry of homotopy continuation, monodromy, witness sets, fiber products, and
numerical irreducible decomposition.

\subsection{Permutation groups}\label{SS:permutation_groups}
Let $\cG\subset S_k$ be a subgroup of the symmetric group on $k$ letters.
Then $\cG$ has a faithful action on $\defcolor{[k]}:=\{1,\dotsc,k\}$.
For $g\in\cG$ and $i\in[k]$, write $g(i)$ for the image of $i$ under $g$.
We say that $\cG$ is \demph{transitive} if for any $i,j\in[k]$ there is an element $g\in\cG$
with $g(i)=j$. 
Every group is transitive on some set, e.g.,~on itself by left multiplication.

The group $\cG$ has an induced action on $s$-tuples, \defcolor{$[k]^s$}.
The action of $\cG$ is \demph{$s$-transitive} if for any two $s$-tuples
$(i_1,\dotsc,i_s)$ and $(j_1,\dotsc,j_s)$ each having distinct elements, 
there is a $g\in\cG$ with $g(i_r)=j_r$ for $r=1,\dotsc,s$.
The full symmetric group $S_k$ is $k$-transitive and its alternating subgroup \defcolor{$A_k$} of even
permutations is $(k{-}2)$-transitive.
There are few other highly transitive groups.
This is explained in~\cite[\S~4]{Cameron} and summarized in the following proposition, which follows from the
O'Nan-Scott Theorem~\cite{OS} and the classification of finite simple groups.

\begin{proposition}[Thm.~4.11~\cite{Cameron}]\label{P:Cameron}
 The only $6$-transitive groups are the symmetric and alternating groups.
 The only $4$-transitive groups are the symmetric and alternating groups, and the Mathieu groups
 $M_{11}$, $M_{12}$, $M_{23}$, and $M_{24}$.
 All $2$-transitive permutation groups are  known.
\end{proposition}

Tables~7.3 and 7.4 in~\cite{Cameron} list the 2-transitive permutation groups.

Suppose that $\cG$ is transitive on $[k]$.
A \demph{block} is a subset $B$ of $[k]$ such that for every $g\in\cG$ either $gB=B$ or $gB\cap B=\emptyset$.
The orbits of a block form a $\cG$-invariant partition of $[k]$ into blocks.
The group $\cG$ is \demph{primitive} if its only blocks are $[k]$ or singletons, otherwise it is
\demph{imprimitive}.
Any 2-transitive permutation group is primitive, and primitive permutation groups that are not symmetric or
alternating are rare---the set of $k$ for which such a nontrivial primitive permutation group exists has density
zero in the natural numbers~\cite[\S~4.9]{Cameron}. 

Each $\cG$-orbit $\cO\subset[k]^2$ determines a graph \defcolor{$\Gamma_\cO$} with vertex set $[k]$---its edges are
the pairs in $\cO$.
For the diagonal orbit $\{(a,a)\mid a\in[k]\}$, this graph is disconnected, consisting of $k$ loops.
Connectivity of all other orbits is equivalent to primitivity (see~\cite[\S~1.11]{Cameron}).

\begin{proposition}[Higman's Theorem~\cite{Higman}]\label{P:Higman}
 A transitive group $\cG$ is primitive if and only if for each non-diagonal orbit $\cO\subset[k]^2$, the 
 graph $\Gamma_\cO$ is connected.
\end{proposition}

Imprimitive groups are subgroups of wreath products $S_a\Wr S_b$ with
$ab=k$ and $a,b>1$
where this decomposition comes from the blocks of a $\cG$-invariant partition.
The dihedral group $D_4$ of the symmetries of a square is isomorphic to $S_2\Wr S_2$, with an imprimitive 
action on the vertices---it preserves the partition into diagonals.
More generally, the dihedral group $D_k$ of symmetries of a regular $k$-gon is imprimitive on the vertices
whenever $k$ is composite.
%
%
%

\subsection{Galois and monodromy groups}\label{SS:Galois_Monodromy}
A map $f\colon V\to U$ between irreducible complex algebraic varieties of the same dimension
with $f(V)$ dense in $U$ is a \demph{dominant} map.
When $f\colon V\to U$ is dominant, the function field $\bC(V)$ of $V$ is a finite extension of $f^*\bC(U)$, the
pullback of the function field of $U$.
This extension has degree $k$, where $k$ is the degree of $f$, which is the cardinality of a general fiber.
The \demph{Galois group $\cG(V{\to}U)$} of $f\colon V\to U$ is the Galois group of the Galois closure of
$\bC(V)$ over $f^*\bC(U)$.

This algebraically defined Galois group is also a geometric monodromy group.
A dominant map $f\colon V\to U$ of equidimensional varieties is a \demph{branched cover}.
The \demph{branch locus $B$} of $f\colon V\to U$ is the set of points
$u\in U$ such that $f^{-1}(u)$ does not consist of $k$ reduced points.
Then $f\colon f^{-1}(U\smallsetminus B)\to U\smallsetminus B$ is a degree $k$ covering space.
The group of deck transformations of this cover is a subgroup of the symmetric group $S_k$ and is
isomorphic to the Galois group $\cG(V{\to}U)$, as permutation groups.
Hermite~\cite{Hermite} realized that Galois and monodromy groups coincide and 
Harris~\cite{Harris79} gave a modern treatment.
The following is elementary.

\begin{proposition}\label{P:liftLoops}
 Let $u\in U\smallsetminus B$.
 Following points in the fiber $f^{-1}(u)$ along lifts to $V$ of loops in $U\smallsetminus B$ gives a
 homomorphism from the fundamental group $\pi_1(U\smallsetminus B)$ of $U\smallsetminus B$ to the
 set of permutations of $f^{-1}(u)$ whose image is the Galois/monodromy group.
\end{proposition}

There is a purely geometric construction of Galois groups using fiber products (explained
in~\cite[\S~3.5]{Va}).
For each $2\leq s\leq k$ let \defcolor{$V^s_U$} be the  $s$-fold fiber product,
\[
   V^s_U\ :=\ \overbrace{V\times_U V\times_U \dotsb \times_U V}^s\,.
\]
We also write \defcolor{$f$} for the map $V^s_U\to U$.
The fiber of $V^s_U$ over a point $u\in U$ is $(f^{-1}(u))^s$, the set of $s$-tuples of points in $f^{-1}(u)$.
Over $U\smallsetminus B$, $V^s_U$ is a covering space of degree $k^s$. 
This is decomposable, and among its components are those lying in the big diagonal \defcolor{$\Delta$}, where
some coordinates of the $s$-tuples coincide.
We define \defcolor{$V^{(s)}$} to be the closure in $V^s_U$ of $f^{-1}(U\smallsetminus B)\smallsetminus\Delta$.
Then every irreducible component of $V^{(s)}$ maps dominantly to $U$ and its fiber over a point 
$u\in U\smallsetminus B$ consists of $s$-tuples of distinct points of $f^{-1}(u)$.
This may be done iteratively as $V^{(s+1)}$ is the union of components of $V^{(s)}\times_U V$ lying outside of
the big diagonal.

Suppose that $s=k$.
Let $u\in U\smallsetminus B$ and write the elements of $f^{-1}(u)$ in some order,
\[
   f^{-1}(u)\ =\ \{v_1,v_2,\dotsc,v_k\}\,.
\]
The fiber of $V^{(k)}$ over $u$ consists of the $k!$ distinct $k$-tuples $(v_{\sigma(1)},\dotsc,v_{\sigma(k)})$ for
$\sigma$ in the symmetric group $S_k$.

\begin{proposition}\label{P:fibre}
 The Galois group $\cG(V{\to}U)$ is the subgroup of $S_k$ consisting of all permutations $\sigma$ such that
 $(v_{\sigma(1)},\dotsc,v_{\sigma(k)})$ lies in the same component of $V^{(k)}$ as does $(v_1,\dotsc,v_k)$.
\end{proposition}

The function field of any component of $V^{(k)}$ is the Galois closure of $\bC(V)$ over $f^*\bC(U)$,
and the construction of $V^{(k)}$ is the geometric counterpart of the usual 
construction of a Galois closure by adjoining successive roots of an irreducible polynomial.
Proposition~\ref{P:fibre} implies that we may read off the Galois group from any irreducible component of
$V^{(k)}$.
In fact $V^{(k-1)}$ will suffice as $V^{(k)}\simeq V^{(k-1)}$.
(Knowing $k{-}1$ points from $\{v_1,\dotsc,v_k\}$ determines the $k$th.)
Other properties of $\cG$ as a permutation group may be read off from these fiber products.

\begin{proposition}\label{P:orbits}
 The irreducible components of $V^{(s)}$ correspond to orbits of $\cG$ acting on 
 \mbox{$s$-tuples} of distinct points.
 In particular, $\cG$ is $s$-transitive if and only if $V^{(s)}$ is irreducible. 
\end{proposition}

\begin{proof}
 This is essentially Lemma~1 of~\cite{SW_double}.
 Let $u\in U\smallsetminus B$ and suppose that $v:=(v_1,\dotsc,v_s)$ and  $v':=(v'_1,\dotsc,v'_s)$ are points in
 the fiber in $V^s_U$ above $u$ that lie in the same irreducible component $X$.
 Let $\sigma$ be a path in $X\smallsetminus f^{-1}(B)$ connecting $v$ to $v'$.
 Then $f(\sigma)=\gamma$ is a loop in $U\smallsetminus B$ based at $u$.
 Lifting $\gamma$ to $V$ gives a monodromy permutation $g\in\cG$ with the property that $g(v_i)=v'_i$ for
 $i=1,\dotsc,s$.
 Thus $v$ and $v'$ lie in the same orbit of $\cG$ acting on $s$-tuples of points of $V$ in the fiber $f^{-1}(u)$.  

 Conversely, let $v_1,\dotsc,v_s\in V$ be points in a fiber above $u\in U\smallsetminus B$ and let $g\in\cG$.
 There is a loop $\gamma\subset U\smallsetminus B$ that is based at $U$ and whose lift to $V$ gives the
 action of $g$ on $f^{-1}(u)$.
 Lifting $\gamma$ to $V^s_U$ gives a path connecting the two points $(v_1,\dotsc,v_s)$ and 
 $(g(v_1),\dotsc,g(v_s))$ in the fiber above $u$, showing that they lie in the same component of $V^s_U$.
 Restricting to $s$-tuples of distinct points establishes the proposition.
\end{proof}

\subsection{Fundamental groups of complements}\label{SS:FundamentalGroups}
Classical algebraic geometers studied the fundamental group $\pi_1(\bP^n\smallsetminus B)$ of the complement of
a hypersurface $B\subset\bP^n$.
Zariski~\cite{Zar} showed that if $\Pi$ is a general two-dimensional
linear subspace of $\bP^n$, then the inclusion 
$\iota\colon \Pi\smallsetminus B\to\bP^n\smallsetminus B$ induces an isomorphism of fundamental groups,
 \begin{equation}\label{Eq:Lefschetz_I}
  \iota_*\ \colon\  \pi_1(\Pi\smallsetminus B)
      \ \xrightarrow{\ \sim\ }\ \pi_1(\bP^n\smallsetminus B)\,.
 \end{equation}
(As the complement of $B$ is connected, we omit base points in our notation.)
Consequently, it suffices to study fundamental groups of complements of plane curves $C\subset\bP^2$.
Zariski also showed that if $\ell$ is a line meeting $B$ in $d=\deg B$ distinct points, so that the intersection
is transverse, then the natural map of fundamental groups
 \[
  \iota_*\ \colon\  \pi_1(\ell\smallsetminus B)
      \ \relbar\joinrel\twoheadrightarrow\ \pi_1(\bP^n\smallsetminus B)
 \]
is a surjection.
(See also~\cite[Prop.~3.3.1]{Dimca}.)

We recall some facts about $\pi_1(\ell\smallsetminus B)$.
Suppose that $B\cap\ell=\{b_1,\dotsc,b_d\}$ and that $p\in\ell\smallsetminus B$ is our base point.
For each $i=1,\dotsc,d$, let $D_i$ be a closed disc in $\ell\simeq\bC\bP^1$ centered at $b_i$ with 
$D_i\cap B=\{b_i\}$.
Choose any path in $\ell\smallsetminus B$ from $p$ to the boundary $\partial D_i$ of $D_i$ and 
let~$\gamma_i$ be the loop based at $p$ that follows that path, traverses the boundary of $D_i$ once
anti-clockwise, and then returns to $p$ along the chosen path.
Any loop in $\ell\smallsetminus B$ based at $p$ that is homotopy-equivalent to $\gamma_i$ (for some choice of
path from $p$ to $\partial D_i$) is a (based) loop in  $\ell\smallsetminus B$ 
\demph{encircling~$b_i$}.
The fundamental group $\pi_1(\ell\smallsetminus B)$ is a free group freely generated by loops encircling any $d{-}1$ points of $B\cap\ell$.
We record the consequence of Zariski's result~that~we~will~use.

\begin{proposition}\label{P:Zariski}
 Let $B\subset\bP^n$ be a hypersurface.
 If $\ell\subset\bP^n$ is any line that meets $B$ in finitely many reduced points, then
 a set of based loops in $\ell$ encircling each of these points generate the fundamental group of the
 complement, $\pi_1(\bP^n\smallsetminus B)$.
\end{proposition}

\subsection{Homotopy continuation and monodromy}
Numerical algebraic geometry~\cite{bertinibook,SW05} uses numerical analysis to study algebraic varieties on a
computer. 
We present its core algorithms of Newton refinement and continuation, and explain how they are used to compute
monodromy. 
Let $F\colon\bC^n\to\bC^n$ be a polynomial map with $F^{-1}(0)$ consisting of finitely many reduced
points.
To any $x\in\bC^n$ that is not a critical point of $F$ so that the Jacobian matrix 
$JF(x)$ of $F$ at $x$ is
nonsingular, we may apply a Newton step
\[
   \defcolor{N_F(x)}\ :=\ x-JF(x)^{-1}\cdot F(x)\,.
\]
If $x$ is sufficiently close to a zero $x^*$ of $F$, then $N_F(x)$ is closer still in that the sequence
defined by $x_0:=x$ and $x_{i+1}:=N_F(x_i)$ for $i\geq 0$ satisfies $\|x^*-x_i\|<2^{1-2^i}\|x^*-x\|$.

A homotopy $H$ is a polynomial map $H\colon\bC^n\times\bC_t\to\bC^n$ that defines a curve $C\subset H^{-1}(0)$
which maps dominantly to $\bC_t$. 
Write $f\colon C\to\bC_t$ for this map.
%
%
We assume that the inverse image $f^{-1}[0,1]$ in $C$ of the interval $[0,1]$ is a collection of arcs connecting
the points of $C$ above $t=1$ to points above $t=0$ which are smooth at $t\neq 0$.
Given a point $(x,1)$ of $C$, standard predictor-corrector methods (e.g.\ Euler tangent prediction followed
by Newton refinement) construct a sequence of points $(x_i,t_i)$ where $x_0=x$ and $1=t_0>t_1>\dotsb>t_s=0$ on
the arc containing $(x,1)$.
This computation of the points in $f^{-1}(0)$ 
from points of $f^{-1}(1)$ 
by continuation along the arcs  $f^{-1}[0,1]$ is called \demph{numerical homotopy continuation}.
Numerical algebraic geometry uses homotopy continuation to solve systems of polynomial equations and to study
algebraic varieties.
While we will not describe methods to solve systems of equations, we will describe other methods of numerical
algebraic geometry.

When $U$ is rational, a branched cover $f\colon V\to U$ gives homotopy paths.
Given a map $g\colon\bC_t\to U$ whose image is not contained in the branch locus $B$ of $f$, the pullback
$g^*V$ is a curve $C$ with a dominant map to $\bC_t$.
Pulling back equations for $V$ gives a homotopy for tracking points of $C$.
We need not restrict to arcs lying over the interval $[0,1]$, but may instead take any path
$\gamma\subset\bC_t$ (or in $U$) that does not meet the branch locus. 
When~$\gamma\subset U\smallsetminus B$ is a loop based at a point $u\in U\smallsetminus B$, homotopy
continuation along $f^{-1}(\gamma)$ starting at~$f^{-1}(u)$ computes the monodromy permutation of $f^{-1}(u)$
given by the homotopy class of $g(\gamma)$ in~$U\smallsetminus B$.
The observation that numerical homotopy continuation may compute monodromy is the 
{\it point~de~d\'epart} of this paper.

\subsection{Numerical algebraic geometry}
Numerical algebraic geometry uses solving, path-tracking, and monodromy to study algebraic varieties on
computer. 
For this, it relies on the fundamental data structure of a witness set, which is a geometric representation
based on linear sections~\cite{INAG,NAG}.

Let $F\colon \bC^n\to \bC^m$ be a polynomial map and suppose that $X$ is a component of $F^{-1}(0)$ of
dimension $r$ and degree $d$.
Let $\cL\colon\bC^n\to\bC^{r}$ be a general affine-linear map so that $\cL^{-1}(0)$ is a general affine
subspace of codimension $r$.
By Bertini's Theorem, $\defcolor{W}:=X\cap\cL^{-1}(0)$ consists of $d$ distinct points, and we call the triple
$(F,\cL,W)$ (or simply $W$) a \demph{witness set} for $X$.
If $\cL$ varies in a family $\{\cL_t\mid t\in\bC\}$, then $V\cap\cL_t^{-1}(0)$ gives a homotopy which may
be used to follow the points of $W$ and sample points of $X$.
This is used in many algorithms to manipulate $X$ based on geometric constructions.
A \demph{witness superset} for $X$ is a finite subset $W'\subset F^{-1}(0)\cap\cL^{-1}(0)$ that contains
$W = X\cap\cL^{-1}(0)$.
In many applications, it suffices to work with a witness superset.
For example, if $X$ is a hypersurface, then $\cL^{-1}(0)$ is a general line, $\ell$, and by Zariski's
Theorem (Prop.~\ref{P:Zariski}), the fundamental group of $\bC^n\smallsetminus V$ is generated by  
loops in $\ell$ encircling the points of $W$, and hence also by loops encircling points~of~$W'$.

One algorithm is computing a witness set for the image of an irreducible variety under a linear map~\cite{HS10}.
Suppose that  $F\colon \bC^n\to \bC^m$ is a polynomial map with $V\subset F^{-1}(0)$ a component of 
dimension $r$ as before, and that we have a linear map $\pi\colon\bC^n\to\bC^p$.
Let $U=\overline{\pi(V)}$ be the closure of the image of $V$ under $\pi$, which we suppose has dimension $q$
and degree $\delta$.
To compute a witness set for $U$ from one for $V$, we need an affine-linear map
$\cL\colon\bC^n\to\bC^{r}$ adapted to the map $\pi$. 

Suppose that $\pi$ is given by $\pi(x)=Ax$ for a matrix $A\in\bC^{p\times n}$.
Let $B$ be a matrix $\left[\begin{smallmatrix} B_1\\B_2\end{smallmatrix}\right]$ where the rows of 
$B_1\in\bC^{q\times n}$ are general vectors in the row space of $A$ and the rows of 
$B_2\in\bC^{(r-q)\times n}$ are general vectors in $\bC^n$.
Then $B_1^{-1}(0)$ is the pullback of a general linear subspace of codimension $q$ in $\bC^p$.
Choose a general vector $v\in\bC^r$, define $\cL(x):=Bx-v$, and set $W:=V\cap\cL^{-1}(0)$.
The quadruple \defcolor{$(F,\pi,\cL,W)$} is a witness set for the image of $V$ under $\pi$.
By the choice of $B$, the number of points in $\pi(W)$ is the degree $\delta$ of $U$, and for~$u\in\pi(W)$, the
number of points in $\pi^{-1}(u)\cap W$ is the degree of the fiber of $V$ over $w$, which has
dimension~$r{-}q$. 
The witness set $(F,\pi,\cL,W)$ for the image $U$ may be computed from any witness set $(F,\cL',W')$  for $V$
by following the points of $W'$ along a path connecting the general affine map $\cL'$ to the special affine map
$\cL$. 

Numerical continuation may be used to sort points in a general affine section of a
reducible variety $V$ into witness sets of its components.
Let $F\colon\bC^n\to\bC^m$ be a polynomial map and suppose that $V=V_1\cup\dotsb\cup V_s$ is a union of
components of $F^{-1}(0)$, all having dimension $r$, and that $\cL\colon\bC^n\to\bC^r$ is a general
affine linear map with $\cL^{-1}(0)$ meeting $V$ transversely in~$d$ points $\defcolor{W}:=V\cap\cL^{-1}(0)$.
The witness sets $\defcolor{W_i}:=V_i\cap\cL^{-1}(0)$ for the components form the 
\demph{witness set partition} of $W$ that we seek.

Following points of $W$ along a homotopy as $\cL$ varies, those from $W_i$ remain on $V_i$.
Consequently, if we compute monodromy by allowing $\cL$ to vary in a loop, the partition of $W$ into orbits
is finer than the witness set partition.
Computing additional monodromy permutations may coarsen this orbit partition, but it will always refine
the witness set partition.

Suppose that $\cL_t$ depends affine-linearly on $t$.
The path of $w\in W$ under the corresponding homotopy will in general be a non-linear
function of $t$. 
However if we follow all points in the witness set $W_i$ for a component, then their sum in $\bC^n$ (the trace) is
an affine-linear function of $t$.
For general $\cL_t$, the only subsets of $W$ whose traces are linear in $t$ are unions of the~$W_i$.
Thus we may test if a union of blocks in an orbit partition is a union of the $W_i$.
These two methods, monodromy break up and the trace test, are combined in the algorithm of numerical
irreducible decomposition~\cite[Ch.~15]{SW05} to compute the witness set partition.

\begin{remark}\label{R:Affine_Charts}
 Oftentimes problems are naturally formulated in terms of homogeneous or multi-homogeneous equations whose
 solutions are subsets of (products of) projective spaces $\bP^n$.
 That is, we have a polynomial map $F\colon \bC^{n+1}\to\bC^r$ and we want to study the projective variety
 given by $F^{-1}(0)$. 
 Restricting $F$ to any affine hyperplane not containing the origin of~$\bC^{n+1}$, we obtain the intersection
 of $F^{-1}(0)$ with an affine chart of $\bP^n$.
 If the hyperplane is general, then the points of interest, including homotopy paths and monodromy loops, will
 lie in that affine chart, and no information is lost by this choice.

 When discussing computation, we will refer to the affine chart given by the vanishing of an affine form, as
 well as referring to the chart via a parameterization of the corresponding affine hyperplane.
 When performing computations, our software works in random affine charts.
\end{remark}

\section{Branch point method}\label{S:Branch}

Given a branched cover $f\colon V\to U$ of degree $k$ with branch locus $B$ such that $U$ is rational, we have
the following direct approach to computing its Galois group $\cG:=\cG(V\to U)$.
Choose a regular value $u\in U\smallsetminus B$ of $f$, so that $f^{-1}(u)$ consists of $k$ reduced points.
Numerically following the points of $f^{-1}(u')$ as $u'$ varies along a sequence of loops in $U\smallsetminus B$
based at $u$ computes a sequence $\sigma_1,\sigma_2,\dotsc$ of monodromy permutations in $\cG\subseteq S_k$.
If $G_i$ is the subgroup of $\cG$ generated by $\sigma_1,\dotsc,\sigma_i$, then we have
\[
    G_1\ \subseteq\ G_2\ \subseteq\ \dotsb\ \subseteq \cG\ \subseteq\ S_k\,.
\]
This method was used in~\cite{LS09} (and elsewhere) to show that $\cG=S_k$ by computing 
enough monodromy permutations so that $G_i=S_k$.
When the Galois group is \demph{deficient} in that $\cG\subsetneq S_k$, then this method cannot compute $\cG$,
for it cannot determine if it has computed generators of $\cG$.
The results described in Section~\ref{S:background} lead to an algorithm to compute a set of generators for
$\cG$ and therefore determine $\cG$.

As $U$ is rational, we may replace it by $\bP^n$ where $n=\dim U$ and assume that $f\colon V\to\bP^n$ is a
branched cover of degree $k$.
The branch locus $B\subset\bP^n$ is the set of points $b\in\bP^n$ where $f^{-1}(b)$ does not consist of $k$
distinct (reduced) points.
As $V$ is irreducible, if $k>1$, then $B$ is a hypersurface.
Suppose that $B$ has degree $d$.

Let $\ell\subset\bP^n$ be a projective line that meets $B$ transversally in $d$ points, so that $W=B\cap\ell$ is
a witness set for $B$.
By Bertini's Theorem, a general line in $\bP^n$ has this property.
Let $u\in \ell\smallsetminus B$ and, for each point $b$ of $B\cap\ell$, choose a loop 
\defcolor{$\gamma_b$} based at $u$ encircling $b$ as in Subsection~\ref{SS:FundamentalGroups}.
Let $\defcolor{\sigma_b}\in S_k$ be the monodromy permutation obtained by lifting $\gamma_b$ to $V$.

\begin{theorem}\label{T:GalGen}
  The Galois group $\cG(V{\to}U)$ is generated by any $d{-}1$ of the monodromy permutations 
  $\{\sigma_b\mid b\in B\cap\ell\}$.
\end{theorem}

\begin{proof}
 By Proposition~\ref{P:liftLoops}, lifting based loops in $\bP^n\smallsetminus B$ to permutations in $S_k$
 gives a surjective homomorphism $\pi_1(\bP^n\smallsetminus B)\to\cG$.
 By Zariski's Theorem (Proposition~\ref{P:Zariski}) the fundamental group of $\bP^n\smallsetminus B$ is generated
 by loops encircling any $d{-}1$ points of $\ell\cap B$, where $d=\deg B$.
 Therefore their lifts to monodromy permutations generate the Galois group $\cG$.
\end{proof}

 It is not necessary to replace $U$ by $\bP^n$.
 If we instead use $\bC^n$ with $B\subset\bC^n$, then $\ell\subset\bC^n$ is a complex line,
 $\ell\simeq\bC$. 
 If $B\cap\ell$ is $d$ distinct points where $d$ is the degree of the closure $\overline{B}$ in~$\bP^n$, then
 the statement of Theorem~\ref{T:GalGen} still holds, as $B\cap\ell=\overline{B}\cap\ell$.

 Lifts of loops encircling the points of a witness superset for $\ell\cap B$ will also generate $\cG$. 

\begin{corollary}\label{C:modifications}
 Suppose that $B'$ is a reducible hypersurface in $\bP^n$ that contains the hypersurface~$B$  
 and that $\ell$ meets $B'$ in a witness superset $W=B'\cap\ell$ for $B$.
 Then lifts $\{\sigma_w\mid w\in W\}$ of loops $\{\gamma_w\mid w\in W\}$ encircling points of $W$ 
 generate $\cG$.
\end{corollary}

\subsection{Branch point algorithm}

Theorem~\ref{T:GalGen} and Corollary~\ref{C:modifications} give a procedure for determining the Galois group $\cG$
of a branched cover $f\colon V\to U$ when $U$ is rational.
Suppose that $V\subset \bP^m\times\bP^n$ is irreducible of dimension $n$ and that the map $f\colon V\to\bP^n$
given by the projection to $\bP^n$ is dominant, so that $f\colon V\to \bP^n$ is a branched cover.

\begin{algorithm}[Branch Point Algorithm]\label{BranchPoint} \ 
\begin{enumerate}
 \item Compute a witness set $\defcolor{W}=B\cap\ell$ (or a witness superset) for the branch locus $B$ of
          $f\colon V\to\bP^n$.

 \item Fix a base point $p\in\ell\smallsetminus B$ and compute the fiber $f^{-1}(p)$. 

 \item Compute monodromy permutations $\{\sigma_w\mid w\in W\}$ that are lifts of based loops in 
        $\ell\smallsetminus B$ encircling the points $w$ of $W$.

\end{enumerate}
  The monodromy permutations $\{\sigma_w\mid w\in W\}$ generate the Galois group $\cG$ of $f\colon V\to\bP^n$.
\end{algorithm}

The correctness of the branch point algorithm follows from Theorem~\ref{T:GalGen} and
Corollary~\ref{C:modifications}. 
We discuss the steps (1) and (3) in more detail.

\subsubsection{Witness superset for the branch locus}\label{SS:branchWitness}

Suppose that $V\subset\bP^m\times\bP^n$ is an irreducible variety of dimension $n$ such that the projection 
$f\colon V\to\bP^n$ is a branched cover with branch locus $B$.
Since $f$ is a proper map (its fibers are projective varieties), $B$ is the set of critical values of $f$.
These are images of the critical points \defcolor{$\CP$}, which are points of 
$V$ where either~$V$ is singular or
it is smooth and the differential of $f$ does not have full rank.

We use $x$ for the coordinates of the cone $\bC^{m+1}$ of $\bP^m$ and $u$ for the cone $\bC^{n+1}$ over $\bP^n$.
Then $V=F^{-1}(0)$, where $F\colon\bC^{m+1}_x\times\bC^{n+1}_u\to\bC^r$ is a system of $r\geq m$ polynomials that
are separately homogeneous in each set of variables $x$ and $u$.
Let $\defcolor{J_xF}:=(\partial F_i/\partial x_j)_{i=1\dotsc,r}^{j=0,\dotsc,m}$ be the $r\times(m{+}1)$-matrix of
the \demph{vertical partial derivatives} of $F$.

\begin{proposition}\label{P:CP_Obvious}
 The critical points $\CP$ of the map $f\colon V\to\bP^n$ are the points of\/ $V$ where $J_xF$ has rank less than
 $m$. 
\end{proposition}

 To compute a witness set for the branch locus $B=f(\CP)$ we will restrict $f\colon V\to\bP^n$ to a line 
 $g\colon\ell\hookrightarrow\bP^n$, obtaining a curve $\defcolor{C}:=g^{-1}(V)\subset \bP^m\times\ell$ 
 equipped with the projection $f\colon C\to\ell$.
 We then compute the critical points on $C$ of this map and their projection to $\ell$.

\begin{example}\label{detExample}
 Consider the irreducible two-dimensional variety \defcolor{$V$} in $\bP_{xy}^1\times\bP_{uvw}^2$ defined by the
 vanishing of $F:=ux^3+vy^3-wxy^2$.  
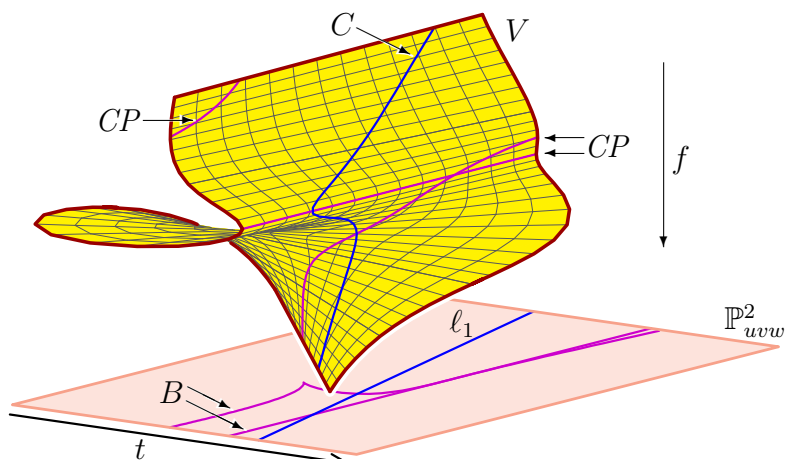
\begin{figure}[htb]
 \begin{picture}(292,170)(2,0)
   \put(2,0){\includegraphics[height=170pt]{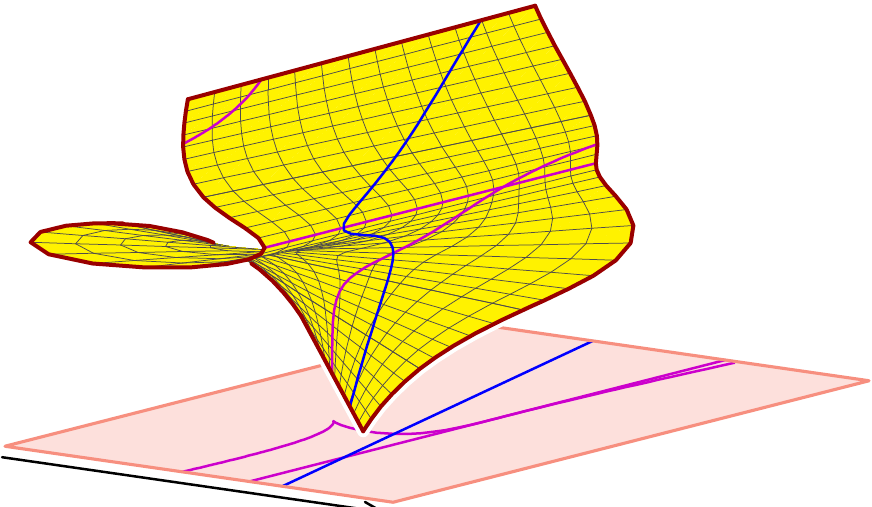}}
   \thicklines
    \put(136,162){\color{white}\vector(2,-1){19}}
    \put(53,128){\color{white}\vector(1,0){20}}
    \put(70,27){\color{white}\vector(2,-1){17}}
    \put(69.5,22){\color{white}\vector(2,-1){22.5}}
  \thinlines
   \put(124,160.5){$C$} \put(134,163){\vector(2,-1){20}}
   \put(190,158){$V$}
   \put(35,124){$\CP$} \put(53,128){\vector(1,0){19}}
   \put(220,114){$\CP$}\put(220,115.5){\vector(-1,0){15}}\put(220,121.5){\vector(-1,0){15}}
   \put(250,150){\vector(0,-1){70}}\put(253,110){$f$}
   \put(273,49){$\bP^2_{uvw}$}
   \put(59,21){$B$}\put(70,27){\vector(2,-1){16}}\put(69.5,22){\vector(2,-1){21.5}}
   \put(169,48){$\ell_1$}
   \put(50, 0){$t$}
 \end{picture}
\caption{The surface $ux^3+vy^3-wxy^2=0$.}
\label{F:cubic_surface}
\end{figure}
 Write $f$ for the projection of $V$ to $\bP^2$, which is a dominant map.
 This has degree three and in Example~\ref{Ex:Branch_loops} we will see that the Galois group is the full
 symmetric group $S_3$. 
 Its critical point locus is the locus of points of $V$ where 
 the Jacobian $J_{xy}F=[\partial F/\partial x \ \partial F/\partial y]$ has rank less than $m=1$.
 This is defined by the vanishing of the partial derivatives 
 as $3F= x\partial F/\partial x +y\partial F/\partial y$.
 Eliminating $x$ and $y$ from the ideal these partial derivatives generate yields the polynomial $u(27uv^2-4w^3)$,
 which defines the branch locus $B$ and shows that both $B$ and $\CP$ are reducible.
 In fact, $B$ consists of the line~$u=0$ and the cuspidal cubic $27uv^2=4w^3$.
 It is singular at the cusp $[1:0:0]$ of the cubic and the point $[0:1:0]$ where the two components meet.
 The cubic has its flex at this point and the line $u=0$ is its tangent at that flex.
 The branch locus is also the discriminant of $F$, considered as a homogeneous cubic in $x,y$.
 We display $V$, $\CP$, and $B$ in Figure~\ref{F:cubic_surface}.

 Consider the line $\ell_1\subset\bP^2$ which is the image of the map $g\colon\bP^1\hookrightarrow\bP^2$ 
 defined by 
 \[
   [s:t]\ \longmapsto\  [s{-}t:2s{-}3t:5s{+}7t]\,.
 \]
 Let $\defcolor{C}\subset\bP^1_{xy}\times\bP^1_{st}$ be the curve $g^{-1}(V)$
 defined by $G:=(s{-}t)x^3+(2s{-}3t)y^3-(5s{+}7t)xy^2$.  
 Its Jacobian with respect to the $x$ and $y$ variables is simply $g^{-1}(J_{xy}F)$, and so the critical points and
 branch locus are pullbacks of those of $F$ along $g$.
 They are defined by the two partial derivatives $\partial G/\partial x$ and $\partial G/\partial y$.
 These equations of bidegree $(1,2)$ have four common zeroes in $\bP^1\times\bP^1$.
 Projecting to $\ell_1=\bP^1_{st}$ and working in the affine chart where $s=1$ yield
 \begin{equation}\label{fourBranch}
    -0.64366 + 0.95874\sqrt{-1}\,,\ -0.18202\,,\  -0.64366 - 0.95874\sqrt{-1}\,,\   1\,.
 \end{equation}
 The first three points lie in the cubic component of $B$, while the last is in the line $u=0$ (so that $s=t$).
 We display the curve $C$ in the real affine chart on $\bP^1_{xy}\times\bP^1_{st}$ given
 by $7x+3y=58$ and $s=1$ as well as the branch locus in the affine chart of $\bC\bP^1_{st}(=\ell_1)$ where $s=1$. 
 This chart for $\bP^1_{xy}$ has parameterization $x=4+9z$ and $y=10-21z$ for $z\in\bC$ and is also used 
 in Figure~\ref{F:cubic_surface}, where the $t$-coordinate is as indicated.
\begin{equation}\label{Eq:Curve_and_branch}
  \raisebox{-50pt}{\begin{picture}(150,100)
     \put(0,0){\includegraphics{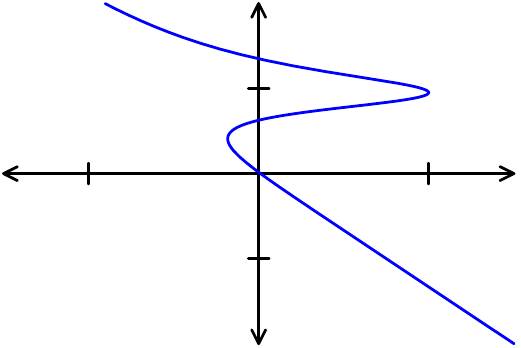}}
     \put(137,53){$t$} \put(78,90){$z$} \put(115,80){$C$}
     \put(63,72){\small$\frac{1}{2}$} \put(121,37){\small$1$}
    \end{picture}
   \qquad\qquad
   \begin{picture}(150,100)
     \put(0,0){\includegraphics{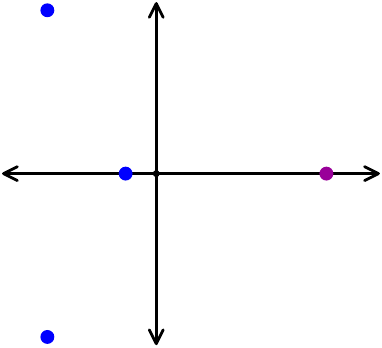}}
     \put(65,53){$\bR$}     \put(90,85){$\bC$}  \put(60,20){$B\cap\ell_1$}
    \end{picture}}
\end{equation}
\end{example}

\begin{remark}\label{randomizationRemark}
 In this example (and, in fact, whenever $V\subset\bP^m\times\bP^n$ is a hypersurface so that~$m=1$), the critical
 point locus $\CP$ is defined by both vertical partial derivatives, and is
 therefore a complete intersection.
 In general, $\CP$ is defined by the polynomial system $F$ and the condition on the rank of the
 Jacobian, and is not a complete intersection.
 In numerical algebraic geometry, it is advantageous to work with complete intersections. 

 There are several methods to reformulate this system as a complete intersection.
 If $r>m$, then $F$ may be replaced by a random subsystem of $m$ polynomials.
 We could also require the vanishing of only a random linear combination of the maximal minors of the Jacobian matrix 
 $J_xF$. 
 Another is to add variables, parameterizing a vector in the null space of $J_xF$.
 That is, add the system $J_xF \cdot v=0$ to $F$, where $v$ spans a general line in $\bC^{m+1}$, so that $v$
 lies in an affine chart $\bC^m$ of $\bP^m$, 
 and then project from $\bC^m_v\times\bP^m_x\times\bP^n_y$ to $\bP^n_y$ and obtain a witness set for $B$.
 This also has the advantage that the new equations $J_xF \cdot v=0$ have total degree equal to those of $F$ and 
 are linear in the entries of $v$.

 These reductions to complete intersections will have not only the points of $B\cap\ell$ as solutions, but possible
 additional solutions, and will therefore compute a witness superset for $B\cap\ell$.
%
%
\end{remark}

\subsubsection{Computing monodromy permutations}\label{step2}
%
%
Suppose that we have a witness superset $W\subset \ell$ for $B$, so that $W$ contains the
transverse intersection $B\cap\ell$.
By Corollary~\ref{C:modifications}, monodromy permutations lifting based loops encircling the points of $W$
generate the Galois group $\cG$.
To compute these encircling loops, we choose a general (random) base point $p\in\ell\smallsetminus W$ and work
in an affine chart of $\ell$ that contains $p$ and $W$ and is identified with $\bC$.
After describing our construction of loops, we will state the condition for genericity.

Let $\epsilon>0$ be any positive number smaller than the minimum distance between points of~$W$.
For $w\in W$, the points $w\pm\epsilon$ and $w\pm\epsilon\sqrt{-1}$ are vertices of a square (diamond)
centered at~$w$ that contains no other points of $W$.
Traversing this anti-clockwise, 
\[
    w+\epsilon\ \leadsto\ w+\epsilon\sqrt{-1}\ \leadsto\ w-\epsilon\ 
       \leadsto\ w-\epsilon\sqrt{-1}\ \leadsto\ w+\epsilon\,,
\]
gives a loop encircling the point $w$.
To obtain loops based at $p$, we concatenate each square loop with a
path from $p$ to that loop as follows.
If $w-p$ has negative imaginary part, then this is the straight line path from $p$ to $w+\epsilon\sqrt{-1}$ and
if $w-p$ has positive imaginary part, the path is from $p$ to $w-\epsilon\sqrt{-1}$.
Our assumption of genericity on the point $p$ is that these chosen paths from $p$ to the squares do not 
meet points of $W$, so that we obtain loops in $\ell\smallsetminus W$. 

Observe that concatenating these loops in anti-clockwise order of the paths from $p$ gives a loop whose
negative encircles the point at infinity.

\begin{example}\label{Ex:Branch_loops}
 We show this collection of based loops encircling the points $W=B\cap\ell_1$ from the
 witness set on the right in~\eqref{Eq:Curve_and_branch} where $p=0.4+0.3\sqrt{-1}$.
\[
   \begin{picture}(90,100)
     \put(0,0){\includegraphics{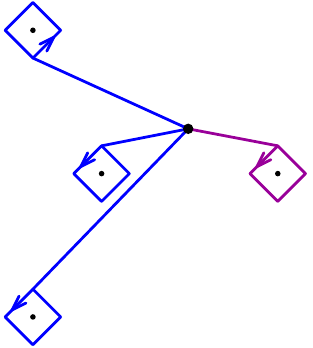}}
     \put(50,70){$p$}
   \end{picture}
\]
 Starting from the rightmost point $1\in W$ and proceeding clockwise, we obtain the permutations 
 $(2,3)$, $(1,3)$, $(1,2)$, and $(1,3)$.
 These generate $S_3$, showing that the Galois group of the cover $V\to\bP^2_{uvw}$ is the
 full symmetric group.
\end{example}

\subsubsection{Implementation subtleties}
In our computations, we do not work directly on projective space, but rather in affine
charts as explained in Remark~\ref{R:Affine_Charts},
and not with general lines, but randomly chosen specific lines.
We illustrate different ways that specific (unfortunate) choices of charts and 
lines may not give a witness set for the branch locus.
While these are overcome in practice by working with affine charts and lines
whose coefficients are randomly generated numbers in~$\bC$, it is important
to point out the subtleties of nongeneric behavior with examples.  

\begin{example}\label{projectiveBranchLocus}
Recall the family $V\to\bP^2$ of cubics in Example~\ref{detExample}. 
The line $\ell_2$ given by the map $[s:t]\mapsto [t{+}s:t{-}s:0]\subset\bP^2$ which
induces a curve $C_2$ that is not
general because the projection $C_2\to\ell_2$ does not have four distinct branch points.  
There are two critical points, each of multiplicity two, as two pairs of simple critical points
came together over $\ell_2$. 
This is observed in Figure~\ref{F:Cubic_critical} where we see that the line $\ell_2$ contains both singular
points $q$ and $r$ of the branch locus, so $B\cap\ell_2$ consists of two points of multiplicity two.
The line $\ell_2$ does not intersect the branch locus $B$ transversally, so Zariski's Theorem
(Proposition~\ref{P:Zariski}) does not hold.
Also, $B\cap\ell_2$ is not a witness set for $B$.
Lifts of loops encircling the points of $B\cap\ell_2$ generate the cyclic group of order three,
rather than the full symmetric group. 
\end{example}

\begin{figure}[htb]
 \begin{picture}(225,100)
  \put(0,0){\includegraphics{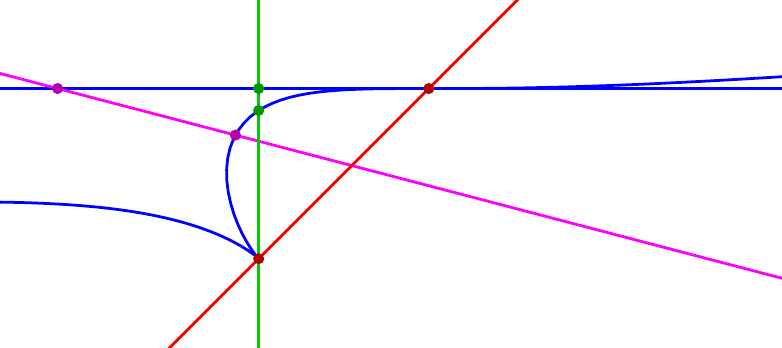}}
  \put(40,6){$\ell_2$}  \put(79,4){$\ell_3$} \put(165,39){$\ell_1$}
  \put(5,45){$B$}  \put(45,78){$B$}
   \put(13,79){$p$}   \put(118,79){$q$}   \put(78,22){$r$}
 \end{picture}
\caption{Branch locus and lines.}
\label{F:Cubic_critical}
\end{figure}
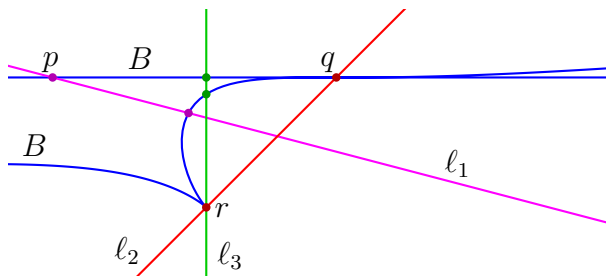

As $V\subset \bP^1\times\bP^2$, we choose affine charts for both factors.
If the charts are not generic, they may omit points of interest.
We illustrate some possibilities.

\begin{example}\label{Ex:affDownstairs}
 Consider the affine chart on $\bP^2$ given by $u=1$, excluding points
 on the one-dimensional component $u=0$ of the branch locus $B$.
 On the line $\ell$ this is the affine chart where $t-s=1$, which omits the fourth point of $B\cap\ell_1$
 of~\eqref{fourBranch}, ($p$ in Figure~\ref{F:Cubic_critical}).
 Thus only three of the four branch points are on this affine chart of $\ell_1$. 
 Since $B$ has degree four, lifting loops encircling these three points gives permutations that generate $\cG$, by Theorem~\ref{T:GalGen}.
\end{example}

\begin{example}\label{Ex:AffineWitness}
 Suppose now that $\ell_3\subset\bP^3$ has equation $v=w$.
 Then $B\cap\ell_3$ consists of three points, with the point $[1:0:0]$ at the cusp of $B$ of multiplicity
 two.
 We may parameterize~$\ell_3$ by $g\colon[s:t]\mapsto[s{-}t:s:s]$.
 Then the affine chart given by $s=1$ does not contain the singular point of $B\cap\ell_3$.
 Even though the intersection is transverse in this affine chart, the two permutations
 we obtain by lifting loops encircling these points do not generate $\cG$, as 
 Theorem~\ref{T:GalGen} does not hold.  
 The difference with Example~\ref{Ex:affDownstairs} is that the branch point at infinity (not on our chosen
 chart) is singular in this case.
\end{example}

\begin{example}\label{affUpstairs}
 A choice of vertical affine chart may also be unfortunate.
 The affine chart on~$\bP^1_{xy}$ where $y=1$ does not meet the line component ($u=y=0$) of the critical point
 locus~$\CP$. 
 Computing a witness set for $\CP$ in this chart and projecting to compute a witness set $B\cap\ell$ for $B$ 
 will only give points in the cubic component of $B$.
 When $\ell$ does not contain the point $q$, this is sufficient to compute $\cG$, for the same reason as in
 Example~\ref{Ex:AffineWitness}.  

 If $V$ is not a hypersurface so that $m>1$, then there may be more interesting components of
 $\CP$ not meeting a given vertical affine chart.
 This may result in the computed points of the witness set $B\cap\ell$ for $B$ being insufficient to generate
 the Galois group $\cG$.
\end{example}

\subsubsection{27 lines on a cubic surface}
A cubic surface $S$ is a hypersurface in $\bP^3$ defined by a homogeneous form of degree three.
It is classical that a smooth cubic surface contains exactly 27 lines (see Figure~\ref{F:27}), and these lines
have a particular incidence structure (see Section~\ref{S:27}).
\begin{figure}[htb]
 \includegraphics[height=150pt]{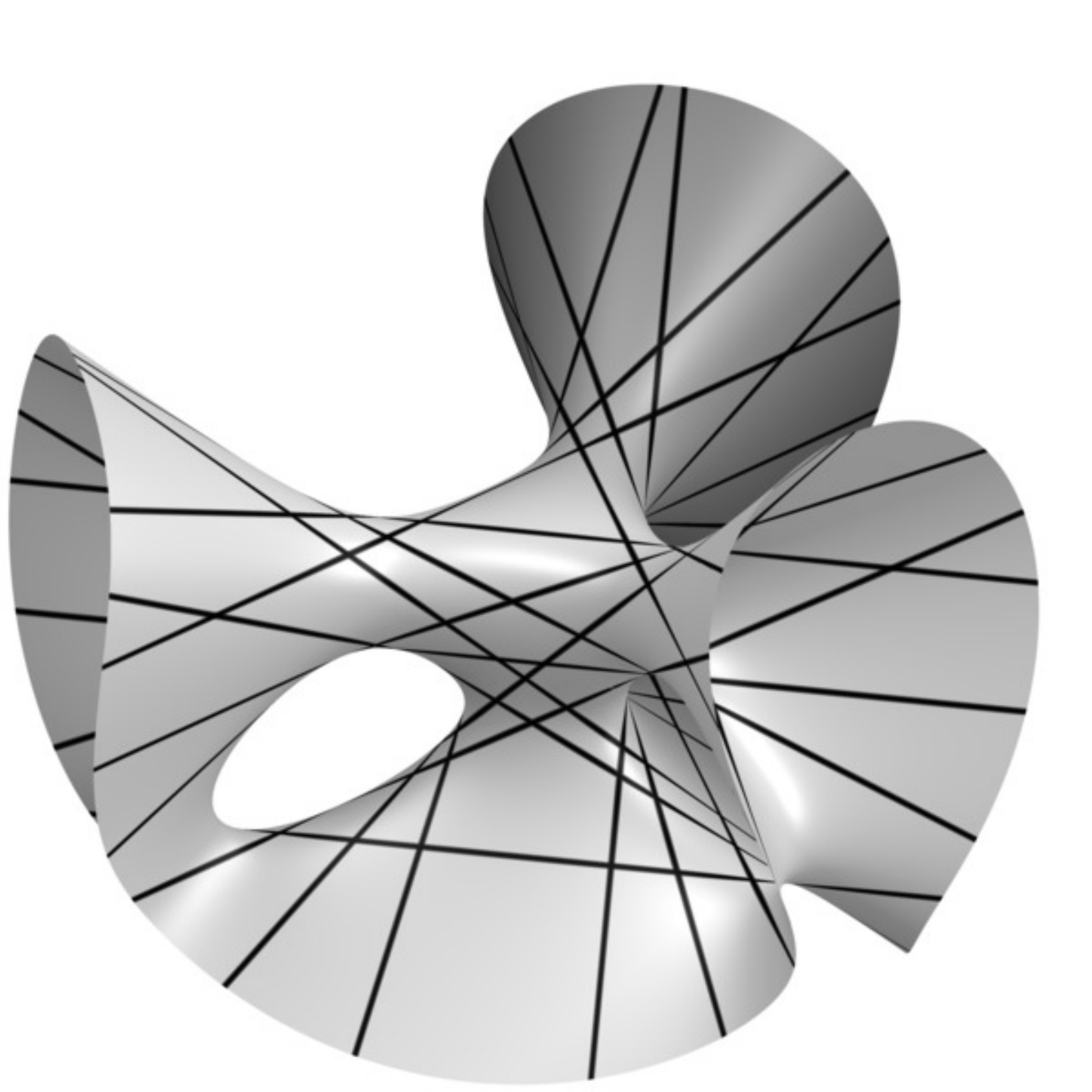}
 \caption{Cubic surface with 27 lines (courtesy of Oliver Labs).} 
 \label{F:27}
\end{figure}
Jordan~\cite{J1870} studied the Galois action on the 27 lines.
It turns out that the Galois group is the group of symmetries of that incidence structure, which is isomorphic
to the Weyl group~$E_6$, a group of order 51,840.
We formulate this as a monodromy problem $f\colon V\to U$ and use the Branch Point Algorithm to compute and
identify this monodromy group.

There are 20 homogeneous cubic monomials in the variables $X,Y,Z,W$ for $\bP^3$, so we identify the space of
cubics with $U=\bP^{19}$.
For $F\in\bP^{19}$, let $\cV(F)$ be the corresponding cubic surface.
Let $\bG$ be the Grassmannian of lines in $\bP^3$, which is an algebraic manifold of dimension 4.
Form the incidence variety
\[
   V\ :=\ \{(F,\ell)\in\bP^{19}\times\bG\mid \ell\mbox{ lies on }\cV(F)\}\,,
\]
which has a map $f\colon V\to\bP^{19}(=U)$.
Algebraic geometry tells us that the cubics with 27 lines are exactly the smooth cubics, and therefore the
branch locus $B$ is exactly the space of singular cubics.
That is, $B$ is given by the classical multivariate discriminant, whose degree was determined by G.~Boole to be
$32$. 

We summarize the computations associated with determining a witness set for this branch locus. 
Let $G$ be a general cubic (variable coefficients) and consider the vectors $\bv=(1,0,k_1,k_2)$ and
$\bw=(0,1,k_3,k_4)$, which span a general line in $\bP^3$.
This line lies on the cubic surface~$\cV(G)$ when the homogeneous cubic  $G(r\bv{+}s\bw)$
(the cubic restricted to the line spanned by $\bv$ and~$\bw$) is identically zero.  
That is, when the coefficients  $K_0,K_1,K_2,K_3$ of $r^3$, $r^2s$, $rs^2$, and $s^3$ in $G(r\bv{+}s\bw)$ vanish.
This defines the incidence variety $V$ in the space $\bP^{19}\times \bC^4_k$, as the
vectors~$\bv,\bw$ and parameters $k_i$ give an affine open chart of $\bG$.
These polynomials $K_i$ are linear in the coefficients of $G$, which shows that the fiber of $V$ above a point
of $\bG$ is a linear subspace of $\bP^{19}$.
Since $\bG$ is irreducible, as are these fibers, we conclude that $V$ is irreducible.

We choose an affine parameterization $g\colon \bC_t\to\ell\subset\bP^{19}$ of a random line $\ell$ in
$\bP^{19}$.
Then $C:=g^*(V)$ is a curve in $\bC_t\times\bC^4_k$ defined by $g^*(K_i)$ for $i=0,\dotsc,3$.
There are $192$ critical points of the  projection $C\to\mathbb{C}_t$, which 
map six-to-one to $32$ branch points. 
Since the branch locus $B$ has degree 32, these branch points are $B\cap\ell$ and form a witness set for $B$.

Computing loops around the $32$ branch points took less than $45$ seconds using our
implementation in {\tt Bertini.m2}~\cite{bertini4M2} using
{\tt Macaulay2} \cite{M2} and {\tt Bertini} \cite{Bertini}. 
This gave  22 distinct permutations, each a product of six 2-cycles.
These are listed in Figure~\ref{F:E6}, and they generate the Weyl group of $E_6$ of order~51,840 confirming
that it is the Galois group of the problem of 27 lines on a cubic surface.
\begin{figure}[htb]
{\small
\[
\begin{array}{rr}
 $(1,3)(4,21)(7,27)(8,23)(9,10)(11,12)$\,,&
 $(1,5)(2,11)(7,13)(8,15)(10,18)(20,21)$\,,\\
 $(1,6)(4,13)(8,25)(10,19)(11,16)(20,27)$\,,&
 $(1,7)(3,27)(5,13)(16,22)(19,24)(25,26)$\,, \\
 $(1,8)(3,23)(5,15)(6,25)(14,22)(17,24)$\,,    &
 $(1,12)(3,11)(13,17)(15,19)(18,25)(20,22)$\,,\\
 $(1,17)(2,27)(8,24)(10,26)(12,13)(16,21)$\,,&
 $(1,18)(4,24)(5,10)(12,25)(14,27)(16,23)$\,,\\
 $(1,19)(2,23)(6,10)(7,24)(12,15)(14,21)$\,,&
 $(1,20)(5,21)(6,27)(9,24)(12,22)(23,26)$\,,\\
 $(1,26)(4,15)(7,25)(10,17)(11,14)(20,23)$\,,&
 $(2,6)(5,16)(8,9)(10,23)(13,22)(17,20)$\,,\\
 $(2,7)(3,17)(4,16)(9,26)(11,13)(23,24)$\,,&
 $(2,8)(3,19)(4,14)(6,9)(11,15)(24,27)$\,, \\
 $(2,12)(3,5)(4,20)(9,18)(13,27)(15,23)$\,, &
 $(2,14)(4,8)(5,26)(13,25)(17,18)(21,23)$\,,\\
 $(2,18)(9,12)(10,11)(14,17)(16,19)(22,24)$\,, &
 $(2,20)(4,12)(6,17)(11,21)(19,26)(24,25)$\,,\\
 $(3,16)(4,17)(6,12)(8,18)(10,15)(22,27)$\,,&
 $(3,18)(5,9)(7,14)(8,16)(11,25)(21,24)$\,,\\
 $(3,26)(8,20)(9,17)(12,14)(15,21)(25,27)$\,,&
 $(6,26)(7,8)(13,15)(14,16)(17,19)(23,27)$\,. 
\end{array}
\] }
\caption{Monodromy permutations.}\label{F:E6}
\end{figure}


\section{Fiber Products}\label{S:fiber}
Let $f\colon V\to U$ be a branched cover of degree $k$ with Galois/monodromy group $\cG$.
As explained in Subsection~\ref{SS:Galois_Monodromy}, the action of $\cG$ on $s$-tuples of points in a fiber of
$f$ is given by the decomposition into irreducible components of iterated fiber products.
We discuss the computation and decomposition of iterated fiber products using 
numerical algebraic~\mbox{geometry}.

Computing fiber products in numerical algebraic geometry was first discussed in~\cite{SW08}.
Suppose that $V\subset\bC_x^m\times\bC_y^n$ is an $n$-dimensional irreducible component of $F^{-1}(0)$ where
$F\colon\bC^m_x\times\bC^n_y\to\bC^m$ and we write $F(x,y)$ with $x\in\bC^m$ and $y\in\bC^n$. 
There are several methods to compute (components of) iterated fiber products.

First, if $\bC^n\simeq\defcolor{\Delta}\subset\bC^n\times\bC^n$ is the diagonal, then 
$V^2_{\bC^n}=V\times_{\bC^n}V\to\bC^n$ is the pullback of the product $V\times V\to\bC^n\times\bC^n$ along
the diagonal $\Delta$.
Were $V$ equal to $F^{-1}(0)$, then $V^2_{\bC^n}$ equals $G^{-1}(0)$, where 
\[
    G\ \colon\ \bC^m\times\bC^m\times\bC^n\ \longrightarrow\ \bC^m\times\bC^m
\]
is given by $G(x^{(1)},x^{(2)},y)= (F(x^{(1)},y), F(x^{(2)},y))$ where $x^{(1)}$ lies in the first copy
of $\bC^m$ and~$x^{(2)}$ lies in the second. 
We also have $V^2_{\bC^n} = (V\times V)\cap(\bC^m\times\bC^m\times\Delta)$.

In general, $V$ is a component of $F^{-1}(0)$ and $V^2_{\bC^n}$ is a union of some components of
$G^{-1}(0)$, and we may compute a witness set representation for $V^2_{\bC^n}$ using its description as
the intersection of the product $(V\times V)$ with $\bC^m\times\bC^m\times\Delta$ as in \S~12.1
of~\cite{bertinibook}.
Iterating this computes~$V^s_{\bC^n}$, which has several irreducible components.
Among these may be components that do not map dominantly to $\bC^n$---these come from fibers of $V\to\bC^n$ of
dimension at least one and thus will lie over a proper subvariety of the branch locus $B$ as $V$ is irreducible
and~$B$ is a hypersurface.
There will also be components lying in the big diagonal where some coordinates in the fiber are equal, with
the remaining components constituting $V^{(s)}$, whose fibers over points of $\bC^n\smallsetminus B$ are
$s$-tuples of distinct points.

In practice, we first restrict $V$ to a general line $\ell\subset\bC^n$, for then $V|_\ell$ will be an
irreducible curve $C$ that maps dominantly to $\ell$.
It suffices to compute the fiber products $C^s_\ell$,  decompose them into irreducible components, and
discard those lying in the big diagonal to obtain $C^{(s)}$ which will be the restriction of $V^{(s)}$ to
$\ell$.
As $C^{(s+1)}$ is the union of components of $C^{(s)}\times_\ell C$ that lie outside the big diagonal, we
may compute $C^{(s)}$ iteratively:
First compute $C^{(2)}$, then for each irreducible component $D$ of $C^{(2)}$, decompose the fiber product
$D\times_\ell C$, removing components in the big diagonal, and continue.
Symmetry may also be used to simplify this computation (e.g., as used in Subsection~\ref{S:27}).

We offer three algorithms based on computing fiber products that obtain information about the Galois/monodromy
group $\cG$ of $f\colon V\to U$.
Let $k$, $\ell$, and $C$ be as above.
Let $p\in \ell\smallsetminus B$ be a point whose fiber in $C$ consists of $k$ distinct  points. 

\begin{algorithm}[Compute $\cG$]\label{Al:Gal_fibre} \ 
 \begin{enumerate}
   \item Compute an irreducible component $X$ of $C^{(k-1)}$.
    \item Fixing a $k{-}1$-tuple $(x_1,\dotsc,x_{k-1})\in X$ lying over $p$, let 
\[
      \cG\ =\ \{\sigma\in S_k\mid (x_{\sigma(1)},\dotsc,x_{\sigma(k-1)})\mbox{ lies over }p\}\,.
\]
  \item Then $\cG$ is the Galois monodromy group.
 \end{enumerate}
\end{algorithm}

\begin{proof}[Proof of correctness]
 Recall that $C^{(k-1)}\simeq C^{(k)}$ since knowing $k-1$ of the points in a fiber of $C$ over a point
 $p\in\ell\smallsetminus B$ determines the remaining point, and the same for $V$.
 Since $X$ lies in a unique component of $V^{(k-1)}$, this follows by Proposition~\ref{P:fibre}.
\end{proof}
\smallskip
\begin{algorithm}[Orbit decomposition of $\cG$ on $s$-tuples and $s$-transitivity]\label{Al:orbit} \ 
  \begin{enumerate}
   \item Compute an irreducible decomposition of $C^{(s)}$,
   \[
      C^{(s)}\ =\ X_1\ \cup\ X_2\ \cup\ \dotsb\ \cup\ X_r\,.
\]
  \item  The action of $\cG$ on distinct $s$-tuples has $r$ orbits, one for each irreducible component
    $X_i$. 
    In the fiber $f^{-1}(p)$ of $C^{(s)}$ these orbits are
\[
    \cO_i\ :=\ f^{-1}(p) \cap X_i \qquad i=1,\dotsc,r\,.
\]
  \item If $r=1$, so that $C^{(s)}$ is irreducible, then $G$ acts $s$-transitively.
 \end{enumerate}
\end{algorithm}

\begin{proof}[Proof of correctness]
 This follows by Proposition~\ref{P:orbits}.
\end{proof}
\smallskip
\begin{algorithm}[Test $\cG$ for primitivity]\label{Al:primitivity}\ 
 \begin{enumerate}
   \item  Compute an irreducible decomposition of $C^{(2)}$.
   \item  If $C^{(2)}$ is irreducible, then $\cG$ is $2$-transitive and primitive.
   \item  Otherwise, use Step~2 of Algorithm~\ref{Al:orbit} to obtain the decomposition of $(f^{-1}(p))^2$ into
     $\cG$-orbits, and construct the graphs $\Gamma_{\cO}$ of Subsection~\ref{SS:permutation_groups}.
   \item Then $\cG$ is primitive if and only if all graphs $\Gamma_{\cO}$ are connected when $\cO$ is not the
     diagonal.
  \end{enumerate}
\end{algorithm}

\begin{proof}[Proof of correctness]
 This follows by Higman's Theorem (Proposition~\ref{P:Higman}).
\end{proof}
\smallskip
\begin{remark}
 Algorithm~\ref{Al:Gal_fibre} to compute $\cG$ using fiber products will be infeasible in practice:
 Even if we have $C\subset\bP^1\times\ell$, then $C^{(k-1)}\subset (\bP^1)^{k-1}\times\ell$, a curve in a
 $k$-dimensional space.
 Such a formulation would have very high degree, as  $C\subset\bP^1\times\ell$ would be defined by a
 polynomial of degree at least $k$.
 Worse than this possibly high dimension and degree of polynomials is that the degree of $C^{(k-1)}\to\ell$
 will be $k!$ with each irreducible component having degree $|\cG|$.
 For the computation in Subsection~\ref{SS:AH}, $k=26$ and $\cG=2^{13}\cdot 13!\approx 5\times 10^{13}$.

 Nevertheless, the interesting transitive permutation groups will fail to be $s$-transitive for~$s\leq 5$
 (Proposition~\ref{P:Cameron}), and interesting characteristics of that action may be discovered through
 studying $C^{(2)}$ using Algorithm~\ref{Al:primitivity}, as shown in the Introduction.
\end{remark}

\subsection{Lines on a cubic surface}\label{S:27}

We briefly review the configuration of the 27 lines on a cubic surface, and what we expect from the
decomposition of $V^{(s)}$ for $s=2,3$.
This is classical and may be found in many sources 
such as \cite[pp.~480--489]{GH}.

Let $p_1,\dotsc,p_6$ be six points in $\bP^2$ not lying on a conic and with no three collinear.
The space of cubics vanishing at $p_1,\dotsc,p_6$ is four-dimensional and gives a rational map
$\bP^2-\to\bP^3$ whose image is a cubic surface $S$ that is isomorphic to $\bP^2$ blown up at the six points
$p_1,\dotsc,p_6$. 
That is, $S$ contains six lines $\widehat{p_1},\dotsc,\widehat{p_6}$ and has a map $\pi\colon S\to\bP^2$ that
sends the line $\widehat{p_i}$ to $p_i$ and is otherwise an isomorphism.
The points of the line  $\widehat{p_i}$ correspond to tangent directions in $\bP^2$ at $p_i$, and the 
\demph{proper transform} of a line or curve in $\bP^2$ is its inverse image under $\pi$, with its tangent
directions at $p_i$ (points in $\widehat{p_i}$) lying above $p_i$, for each $i$.
This surface $S$ contains~27 lines as follows.
\begin{itemize}

 \item Six are the blow ups $\widehat{p_i}$ of the points $p_i$ for $i=1,\dotsc,6$.

 \item Fifteen $(=\binom{6}{2})$ are the proper transforms $\widehat{\ell_{ij}}$ of the lines through two points
   $p_i$ and~$p_j$ for $1\leq i<j\leq 6$.

 \item Six are the proper transforms $\widehat{C_i}$ of the conics through five points 
        \mbox{$\{p_1,\dotsc,p_6\}\smallsetminus\{p_i\}$} for $i=1,\dotsc,6$.

\end{itemize}
Figure~\ref{F:six_points} gives a configuration of six points in $\bP^2$, together with
three of the lines and one of the conics they determine, showing some points of intersection.
\begin{figure}[htb]
 \begin{picture}(160,120)(0,8)
  \put(0,8){\includegraphics{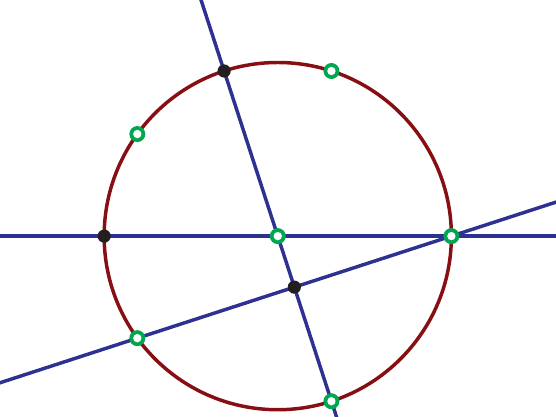}}
  \put(0,64){$\ell_{12}$}  \put(45,118){$\ell_{16}$}
  \put(0,26){$\ell_{25}$}  \put(118,95){$C_1$}
  \put(31,96){$p_4$}  \put(91,114){$p_3$}
  \put(83,66){$p_1$}  \put(118,65){$p_2$}
  \put(32,21){$p_5$}  \put(96,21){$p_6$}
 \end{picture}
\caption{Six points, some lines, and a conic.}
\label{F:six_points}
\end{figure}
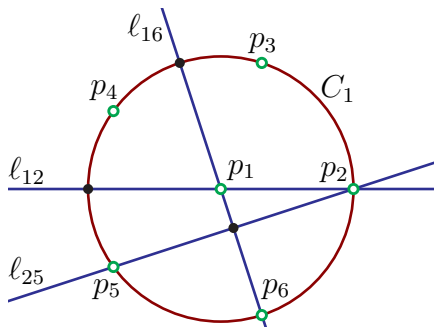

Each line $\lambda$ on $S$ is disjoint from 16 others and meets the remaining ten.
With these ten,~$\lambda$ forms five triangles---the plane $\Pi$ containing any two lines $\lambda,\mu$ on $S$
that meet will contain a third line $\nu$ on $S$ as $\Pi\cap S$ is a plane cubic curve containing $\lambda$ and
$\mu$.

We explain this in detail for the lines $\widehat{p_1}$, $\widehat{\ell_{12}}$, and $\widehat{C_1}$.
\begin{itemize}
 \item
      The line $\widehat{p_1}$ is disjoint from $\widehat{p_i}$ for $2\leq i\leq 6$ as the points are distinct.
      It is disjoint from $\widehat{\ell_{ij}}$ for $2\leq i<j\leq 6$, as no such line $\ell_{ij}$ meets
      $p_1$, and it is disjoint from $\widehat{C_1}$, as $p_1\not\in C_1$.
      The line $\widehat{p_1}$ does meet the lines $\widehat{C_i}$ and $\widehat{\ell_{1i}}$ for
      $2\leq i\leq 6$, as $p_1$ lies on these conics $C_i$ and lines $\ell_{1i}$.

  \item 
      The line $\widehat{\ell_{12}}$ is disjoint from the lines $\widehat{p_i}$,  $\widehat{\ell_{1i}}$,
      $\widehat{\ell_{2i}}$, and $\widehat{C_i}$, for $3\leq i\leq 6$.
      We have seen this for the $\widehat{p_i}$.
      For the lines, $\widehat{\ell_{1i}}$ and $\widehat{\ell_{2i}}$, this is because $\ell_{12}$ meets the lines
      $\ell_{1i}$ and $\ell_{2i}$ at the points $p_1$ and $p_2$, but it has a different slope at each point,
      and the same is true for the conic $C_i$.
      We have seen that $\widehat{\ell_{12}}$ meets both $\widehat{p_1}$ and $\widehat{p_2}$.
      It also meets $\widehat{\ell_{ij}}$ for $2\leq i<j\leq 6$, as well as $\widehat{C_1}$, and
      $\widehat{C_2}$, because $\ell_{12}$ meets the underlying lines and conics at points outside of
      $p_1,\dotsc,p_6$.
      (See Figure~\ref{F:six_points}.)

 \item Finally, the line $\widehat{C_1}$ is disjoint from $\widehat{p_1}$, from $\widehat{\ell_{ij}}$ for
     $2\leq i<j\leq 6$, and from $\widehat{C_i}$ for $2\leq i\leq 6$.
     The last is because $C_1$ meets each of those conics in four of the points $p_2,\dotsc,p_6$ and no other
     points.
     As we have seen, $\widehat{C_1}$ meets $\widehat{p_i}$ and $\widehat{\ell_{1i}}$ for $2\leq i\leq 6$.

\end{itemize} 

We describe the decomposition of $V^{(2)}$ and $V^{(3)}$.
Let \defcolor{$V^{[2]}$} be the closure in $V^2_{\bP^{19}}$ of its restriction to $\bP^{19}\smallsetminus B$.
Let $p\in\bP^{19}\smallsetminus B$.
The fiber $f^{-1}(p)$ in $V^{[2]}$ consists of the $27^2=729$ pairs $(\lambda,\mu)$ of lines $\lambda,\mu$
that lie on the cubic given by $p$.
The variety $V^{[2]}$ has degree $729$ over~$\bP^{19}$ and decomposes into three subvarieties.
We describe typical points $(\lambda,\mu)$ in the fibers of~each.
\begin{enumerate}

 \item The diagonal $\Delta$, whose points are pairs where $\lambda=\mu$.
        It has degree 27, is irreducible and isomorphic to $V$.

 \item The set of disjoint pairs, $D$, whose points are pairs of disjoint lines $(\lambda,\mu)$ where
        $\lambda\cap\mu=\emptyset$. 
        It has degree $27\cdot 16=432$ over $\bP^{19}$.

 \item The set of incident pairs, $I$, whose points are pairs of incident lines  $(\lambda,\mu)$ where
   $\lambda\cap\mu\neq\emptyset$. 
         It has degree $27\cdot 10=270$ over $\bP^{19}$.

\end{enumerate}
In particular, since $V^{(2)}$ decomposes into two components, which we verified using
a numerical irreducible decomposition via {\tt Bertini} \cite{Bertini},
the action of $\cG$ fails to be 2-transitive. 

However, $\cG$ is primitive, which may be seen using Algorithm~\ref{Al:primitivity} and Higman's Theorem
(Proposition~\ref{P:Higman}). 
As $V$ is irreducible, $\cG$ is transitive.
Since $D$ is irreducible, the 216 unordered pairs of disjoint lines form an orbit \defcolor{$\cD$} of $\cG$.
The graph $\Gamma_\cD$ is connected.
Indeed, the only non-neighbors of $\widehat{p_1}$ are $\widehat{C_j}$ and 
$\widehat{\ell_{1j}}$ for $2\leq j\leq 6$.
As $\widehat{C_j}$ is disjoint from $\widehat{C_1}$ and $\widehat{\ell_{1j}}$ is disjoint from $\widehat{p_i}$
for $i\neq 1,j$, and $\widehat{p_1}$ is disjoint from both $\widehat{C_1}$ and $\widehat{p_i}$,
$\Gamma_\cD$ is connected (and has diameter two).
Similarly, as $I$ is irreducible, the pairs of incident lines form a single orbit whose associated graph may
be checked to have diameter two.

The decomposition of $V^{(3)}$ has eight components, 
which we verified using a numerical irreducible decomposition via {\tt Bertini}~\cite{Bertini}.
These components have four different types up to the action of $S_3$ on triples.
\begin{enumerate}
  \item Triangles, $\tau$.
         The typical point of $\tau$ is a triangle, three distinct lines that meet each other.
         This has degree $270$ over $\bP^{19}$ and is a component of $I\times_{\bP^{19}}V$.

   \item Mutually skew triples, $\sigma$.
           The typical point of $\sigma$ is three lines, none of which meet each other.
           This has degree $4320$ over $\bP^{19}$, and is a component of $D\times_{\bP^{19}}V$.

   \item There are three components $\rho_i$ consisting of triples $(\lambda_1,\lambda_2,\lambda_3)$ of lines
     where the $i$th  line does not meet the other two, but those two meet.
     Each has degree $2160$ over $\bP^{19}$ and~$\mu_3$ is a component of $I\times_{\bP^{19}}V$.

   \item There are three components $\xi_i$ consisting of triples $(\lambda_1,\lambda_2,\lambda_3)$ of lines
     where the $i$th line meets the other two, but those two do not meet.
     Each has degree $2160$ over $\bP^{19}$ and~$\mu_3$ is a component of $D\times_{\bP^{19}}V$.

\end{enumerate}

\section{Galois groups in applications}\label{S:examples}
We present three problems from applications that have interesting Galois groups, which we
compute using our methods.

\subsection{Formation shape control}\label{SS:AH}
Anderson and Helmke~\cite{AH14} consider a least-squares solution to a problem of placing agents at
positions $x_1,\dotsc,x_N\in\bR^d$ having preferred pairwise distances $u_{ij}=u_{ji}$ for $1\leq i,j\leq N$,
that is, minimizing the potential
\[
    \defcolor{\Psi_u}\ :=\ \sum_{i,j} \left(\left\Vert x_{i}-x_{j}\right\Vert ^{2}-u_{ij}^{2}\right)^{2}\,.
\]
They specialize to points on a line $d=1$ and eliminate translational ambiguity by setting $x_N=0$.
Then they relax the problem to finding the complex critical points of the gradient descent flow given by
$\Psi_u$.
This yields the system of cubic equations
\[
   0\ =\ \sum_{j=1}^N \bigl((x_i-x_j)^2-u_{ij}^2\bigr)(x_i-x_j)
  \quad i=1,\dotsc,N{-}1\,,\qquad x_N=0\,.
\]
Thus when $N\geq 4$ there are at most $3^{N-1}$ isolated complex solutions for general $u_{ij}$, one of which
is degenerate: $x_i=0$ for all $i$ with the agents collocated at the origin.
When the $u_{ij}$ are real, there are always at least $2N{-}1$ real critical formations.
The symmetry $x_i\mapsto -x_i$ reflecting in the origin gives an involution acting freely on the nondegenerate
solutions. 
This commutes with with complex conjugation and implies that there is
an additional congruence modulo four in the number of real solutions (compare to~\cite{HSZII,HSZI}).

We compute the Galois group when $N=4$.
Anderson and Helmke show that the upper bound of 27 critical points is obtained for general
$u_{ij}$, with 26 nondegenerate solutions having no two agents collocated.
They also show that all possible numbers of real critical points (not including the origin),
$6,10,14,18,22,26$ between $6=2N-2$ and $26$ that are congruent to $6$ modulo four do indeed occur.
The symmetry  $x_i\mapsto -x_i$ implies that the Galois group preserves the partition of the solutions into
the pairs $\{x_i,-x_i\}$, which implies that it is a subgroup of the wreath product $S_2\Wr S_{13}$, which has order
$\mbox{51,011,754,393,600} =2^{13}\cdot 13!$. 

The Branch Point Algorithm shows that the Galois group of this system is indeed 
equal to the wreath product $S_2\Wr S_{13}$. 
We found this by computing $144$ critical points that map two-to-one to the $72$ branch points. 
Taking loops around each of the $72$ branch points can be performed in under a minute using one processor on a
laptop. 
Interestingly, while most critical points were simple in that their local monodromy was a 2-cycle, several
were not. 
 
\subsection{Alt-Burmester 4-bar examples}

In 1886, Burmester \cite{Burmester} considered the synthesis problem for planar four-bar linkages based on
motion generation, specifying  poses along a curve.  
Alt \cite{Alt} proposed synthesis problems based on path generation,
specifying positions along a curve.  The synthesis problem consisting
of some poses and some positions was called an Alt-Burmester problem in \cite{AltBurmesterOrig}
with the complete solution to all Alt-Burmester problems described in \cite{AltBurmester}.  
We compute the Galois group for four of the Alt-Burmester problems having finitely many solutions.

Figure~\ref{F:AltBurmester} illustrates these problems.
A \demph{four-bar linkage} is a quadrilateral with one side fixed and four rotating joints.
A triangle is erected on the side opposite the fixed side, and a tool is mounted on the
apex of the triangle with a particular orientation.
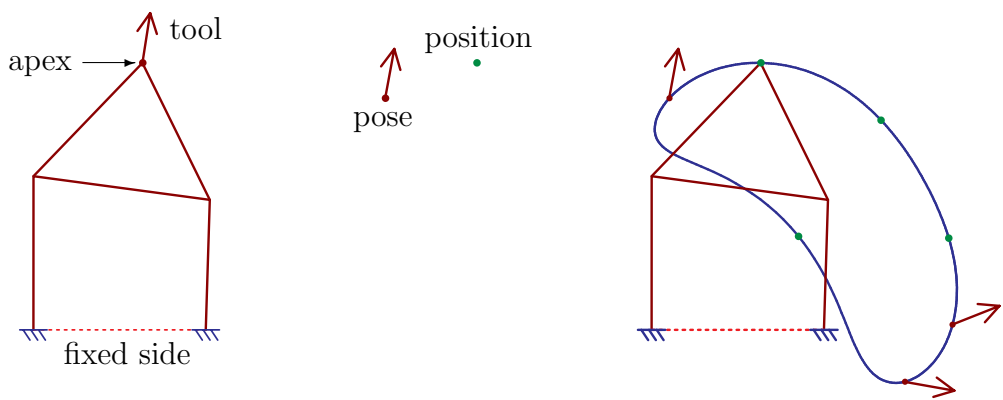
\begin{figure}[htb]
  \begin{picture}(100,150)(-1,0)
   \put(0,0){\includegraphics[height=150pt]{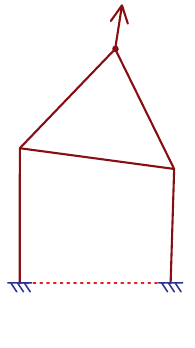}}
   \put(20,15){fixed side}  
   \put(60,138){tool}  
   \put(-1,126){apex \raisebox{3pt}{\vector(1,0){20}}}
  \end{picture}
  \qquad
  \begin{picture}(80,150)
   \put(0,0){\includegraphics[height=150pt]{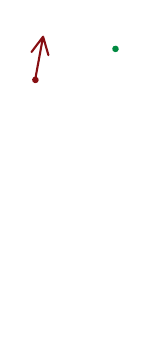}}
   \put(30,135){position}  
   \put(3,105){pose}
  \end{picture}
  \qquad
  \begin{picture}(142.5,150)
   \put(0,0){\includegraphics[height=150pt]{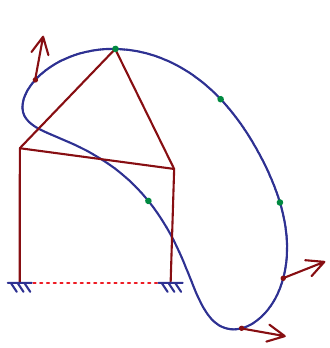}}
  \end{picture}
 \caption{A linkage, poses, positions, and a solution for 3 poses and 4 positions.}
 \label{F:AltBurmester} 
\end{figure}
A pose is a position together with an orientation for the tool.
Specifying $M$ poses and $N=10-2M$ positions, there will generically
be finitely many linkages that take on the given
poses and whose apex can pass through the given positions in its motion.

Following \cite{AltBurmester} in isotropic coordinates, 
the $M$-pose and $N$-position Alt-Burmester problem is described by the following parameters:
\begin{eqnarray*}
   \mbox{positions:}&& (D_j,\overline{D}_j)\,,\ \mbox{for } j = 1,\dotsc,M{+}N\\
   \mbox{orientations:}&&
    (\Theta_j,\overline{\Theta}_j)\,,\ \mbox{for }\  j = 1,\dotsc,M 
    \ \mbox{with}\ \Theta_j\overline{\Theta}_j = 1\,.
\end{eqnarray*}
With variables $G_1,G_2,z_1,z_2,\overline{G}_1,\overline{G}_2,\overline{z}_1,\overline{z}_2,
\Theta_j,\overline{\Theta}_j$ for $j = M{+}1,\dots,M{+}N$, we consider
 \begin{eqnarray*}
   L_{rj}\overline{L}_{rj} - L_{r1}\overline{L}_{r1} &=& 0\,,\ \mbox{for }
          j = 2,\dots,M+N\ \mbox{and } r = 1,2 \\
    \Theta_j\overline{\Theta}_j - 1 & = & 0\,,\ \mbox{for } j = M+1,\dots,M+N
 \end{eqnarray*}
%
%
where
\[
   L_{rj}\ :=\ \Theta_j z_r + D_j - G_r\ \mbox{\ and\ }\ 
   \overline{L}_{rj}\: =\ \overline{\Theta}_j \overline{z}_r + \overline{D}_j - \overline{G}_r\,.
\]

We first consider the classical case studied by Burmester, namely $M = 5$ and $N = 0$.  
As noted by Burmester, the system of $8$ polynomials, using modern terminology,
is a fiber product since the synthesis problem for four-bar linkages uncouples
into two synthesis problems for so-called RR dyads (left- and right- halves of the linkage).  
Each RR dyad synthesis problem
has $4$ solutions with Galois group~$S_4$.  Thus, the polynomial system for the four-bar 
linkage synthesis problem has $16$ solutions which decomposes into
two components: $4$ points on the diagonal $\Delta$ and $12$ disjoint pairs.
The Branch Point Algorithm uses homotopy continuation to track a loop around each of the $64$ branch points. 
These loops yield the permutations listed in Figure~\ref{burm50Permutations}. 
\begin{figure}[htb]
{\small
\[
\begin{array}{rr}
 ${\bf (1,2,3,4)}    (5,8,13,15)(6,10,12,16)(7,9,11,14)$\,,    &
 ${\bf (3,4)   }(5,14)(6,15)(7,16)(8,9)(12,13)$\,,          \\
 ${\bf  (1,2,4,3)}   (5,16,8,11)(6,14,9,12)(7,15,10,13)$\,,  &
 ${\bf  (2,3)}   (5,7)(8,10)(11,15)(12,14)(13,16)$\,,       \\
 ${\bf  (1,3,2,4)}   (5,9,15,12)(6,8,14,13)(7,10,16,11)$\,,  &
 ${\bf  (1,3)}   (6,7)(8,15)(9,16)(10,14)(11,12)$\,,        \\
 ${\bf(1,4,2,3)}   (5,12,15,9)(6,13,14,8)(7,11,16,10)$\,,    &
 ${\bf(1,2)}   (5,6)(8,12)(9,13)(10,11)(14,15)$\,,          \\
 ${\bf (1,3)(2,4)}   (5,13)(6,12)(7,11)(8,15)(9,14)(10,16)$\,,      &
 ${\bf (2,4)}   (5,13)(6,11)(7,12)(9,10)(14,16)$\,,      \\
 ${\bf (1,4)(2,3)}   (5,8)(6,9)(7,10)(11,16)(12,14)(13,15)$\,,  &
 ${\bf (1,4)}   (5,10)(6,9)(7,8)(11,13)(15,16)$\,,      \\
 ${\bf (1,4,3)}   (5,14,10)(6,16,8)(7,15,9)(11,13,12)$\,,      &
 ${\bf (2,3,4)}   (5,12,16)(6,11,15)(7,13,14)(8,9,10)$\,,     \\
 ${\bf (1,2,3)}   (5,7,6)(8,11,14)(9,13,16)(10,12,15)$\,,      &
 ${\bf (1,2,4)}   (5,9,11)(6,10,13)(7,8,12)(14,16,15)$\,,     \\
 ${\bf (2,4,3)}   (5,16,12)(6,15,11)(7,14,13)(8,10,9)$\,,     &
 ${\bf (1,3,4)}   (5,10,14)(6,8,16)(7,9,15)(11,12,13)$\,,     
\end{array}
\] }
\caption{Monodromy permutations for Burmester 5-0.}\label{burm50Permutations}
\end{figure}
Cycles involving the first four solutions are in {\bf boldface}, to help see that these
solutions are permuted amongst themselves while the other twelve solutions are permuted
amongst themselves.  
This shows that the Galois group of each component and of their union 
is also~$S_4$. 
For the off-diagonal component, it is the action of $S_4$ on ordered pairs of numbers $\{1,2,3,4\}$.

The remaining three cases under consideration are $(M,N) = (4,2),(3,4),(2,6)$ 
which have $60$, $402$, and $2224$ isolated solutions, respectively \cite[Table~1]{AltBurmester}.
In each, the left-right symmetry of the mechanism ($r=1,2$ above) implies that the Galois group of a problem
with $k=2m$ solutions will be a subgroup of the group $S_2\Wr S_m$ of order $2^m m!$.
We applied the Branch Point Algorithm first to the Alt-Burmester problem with $M=4$ and $N=2$. 
We tracked a loop around each of the 2094 branch points to compute generators of the Galois group, 
thereby showing the Galois group has order 
\[
  284\,813\,089\,515\,958\,324\,736\,640\,819\,941\,867\,520\,000\,000\  =\ 2^{30} \cdot 30!\,,
\]
and is thus the full wreath product $S_2\Wr S_{30}$.
This Galois group is the largest it could be given the left-right symmetry. 

For each of the cases when $(M,N)$ is $(3,4)$ and $(2,6)$ we computed ten random permutations,
which was sufficient to show that the Galois groups of these problems are indeed equal to 
$S_2\Wr S_{201}$ and $S_2\Wr S_{1112}$ having order
$2^{201}\cdot 201! \approx 5\cdot 10^{437}$ and 
$2^{1112}\cdot 1112! \approx 10^{3241}$, respectively.

\subsection{Algebraic statistics example}
Maximum likelihood estimation on a discrete algebraic statistical model $\cM$ involves maximizing the
likelihood function $\defcolor{\ell_u(p)}:=p_0^{u_0}p_1^{u_1}\cdots p_n^{u_n}$ for data consisting of positive
integers $u_0,\dots,u_n$ restricted to the model. 
The model $\cM$ is defined by  polynomial equations in the probability
%
%
simplex, which is the subset of $\bR^{n+1}$ where $p_0+\dotsb+p_n=1$ and $p_i\geq 0$.
We consider the  Zariski closure of $\cM$ in $\bP^n$ (also written~$\cM$), as $p_0+\cdots+p_n=1$ defines an
affine open subset of $\bP^n$.

The variety $V$ of critical points of $\ell_u$ on the model $\cM$ lies in $\bP^n_p\times\bP^n_u$.
This is the Zariski closure of points $(p,u)$ where $p$ is a smooth point of $\cM$ and a critical point of
$\ell_u$.  
Then $V$ is $n$-dimensional  and irreducible, and its projection to $\bP^n_u$ gives a branched cover
whose degree is the \demph{maximum likelihood degree}~\cite{HKS05,HS14}. 
For an algebraic statistical model, we can ask for the Galois group of this maximum likelihood estimation (the
branched cover $V\to\bP^n$). 

The model defined by the determinant~\eqref{rank2Sym}
has maximum likelihood degree $6$, and has a Galois group that is a proper subgroup of the full symmetric
group $S_6$~\cite{HRS14}. 
 \begin{equation}\label{rank2Sym}
   \det\;\begin{pmatrix}2p_{11}&p_{12}&p_{13}\\
          p_{21}&2p_{22}&p_{23}\\p_{13}&p_{23}&2p_{33}\end{pmatrix}\ =\ 0\,.
 \end{equation}
Using the Branch Point Algorithm, we solve a system of equations to find $24$ critical points of the projection (note the difference between critical points of the likelihood function and critical points of the projection).
The critical points of the projection map $2$ to $1$ to a set of $12$ branch points
yielding a witness set for the branch point locus which is a component of the data discriminant\footnote{
The defining polynomial of degree $12$ was computed in~\cite[Ex.~6]{RT15} and is available at
the website
\url{https://sites.google.com/site/rootclassification/publications/DD}.}.
The Branch Point Algorithm finds the following generating set of the Galois group,
which has order $4! = 24$ and is isomorphic to $S_4$:
\[
  \left\{\begin{array}{cccccc}
   %
   (12)(34), &   
   (26)(45), &   
   (14)(23), &   
   (15)(36), &   
   (16)(35), &   
   (126)(345)    
  \end{array}\right\}\,.
\]
The reason for the interesting Galois group is explained by maximum likelihood duality \cite{DR14}. 
Moreover, in \cite[\S~5]{RT15}, it is shown that over a real data point, a typical fiber has 
either~$2$~or~$6$ real points. 
This further strengthens the notion that degenerate Galois groups can help 
identify the possibility of interesting real structures. 

\section{Conclusion}\label{S:conclusion}

We have given algorithms to compute Galois groups.
The main contributions are two numerical algorithms [Algorithm \ref{BranchPoint} and \ref{Al:orbit}], that
allow for practical computation of Galois groups.  
The first algorithm, the Branch Point Algorithm, has been implemented in 
{\tt Macaulay2} building on monodromy computations performed by
{\tt Bertini} and is publicly available\footnote{\url{http://home.uchicago.edu/~joisro/quickLinks/NCGG/}}. 
Moreover, we have shown its effectiveness in examples ranging from enumerative geometry, kinematics, and statistics. 
The other algorithm uses fiber products to test for $s$-transitivity.
This is practical as permutation groups that are not alternating or symmetric are at most $5$-transitive (and
$k\leq 24$).
These two algorithms demonstrate that homotopy continuation can be used to compute Galois groups. 

%
\bibliographystyle{plain}


\begin{thebibliography}{10}

\bibitem{Alt}
H.~Alt.
\newblock \"{U}ber die erzeugung gegebener ebener kurven mit hilfe des
  gelenkviereckes.
\newblock {\em Zeitschrift f\"{u}r Angewandte Mathematik und Mechanik},
  3(1):13--19, 1923.

\bibitem{AH14}
B.D.O. Anderson and U.~Helmke.
\newblock Counting critical formations on a line.
\newblock {\em SIAM Journal on Control and Optimization}, 52(1):219--242, 2014.

\bibitem{bertini4M2}
D.J. Bates, E.~Gross, A.~Leykin, and J.I. Rodriguez.
\newblock {Bertini for Macaulay2}.
\newblock {\em preprint arXiv:1310.3297}, 2013.

\bibitem{Bertini}
D.J. Bates, J.D. Hauenstein, A.J. Sommese, and C.W. Wampler.
\newblock Bertini: Software for numerical algebraic geometry.
\newblock Available at \url{http://bertini.nd.edu}.

\bibitem{bertinibook}
D.J. Bates, J.D. Hauenstein, A.J. Sommese, and C.W. Wampler.
\newblock {\em Numerically solving polynomial systems with {B}ertini},
  volume~25 of {\em Software, Environments, and Tools}.
\newblock Society for Industrial and Applied Mathematics (SIAM), Philadelphia,
  PA, 2013.

\bibitem{AltBurmester}
D.A. Brake, J.D. Hauenstein, A.P. Murray, D.H. Myszka, and C.W. Wampler.
\newblock The complete solution of {A}lt-{B}urmester synthesis problems for
  four-bar linkages.
\newblock {\em Journal of Mechanisms and Robotics}, 8(4):041018, 2016.

\bibitem{BdCS}
C.~Brooks, A.~Mart\'in~del Campo, and F.~Sottile.
\newblock Galois groups of {S}chubert problems of lines are at least
  alternating.
\newblock {\em Trans. Amer. Math. Soc.}, 367:4183--4206, 2015.

\bibitem{Burmester}
L.~Burmester.
\newblock {\em Lehrbuch der Kinematic}.
\newblock Verlag Von Arthur Felix, Leipzig, Germany, 1886.

\bibitem{Cameron}
P.J. Cameron.
\newblock {\em Permutation groups}, volume~45 of {\em London Mathematical
  Society Student Texts}.
\newblock Cambridge University Press, Cambridge, 1999.

\bibitem{Dimca}
A.~Dimca.
\newblock {\em Singularities and topology of hypersurfaces}.
\newblock Universitext. Springer-Verlag, New York, 1992.

\bibitem{DR14}
J.~Draisma and J.~I. Rodriguez.
\newblock Maximum likelihood duality for determinantal varieties.
\newblock {\em International Mathematics Research Notices},
  2014(20):5648--5666, 2014.

\bibitem{M2}
D.R. Grayson and M.E. Stillman.
\newblock Macaulay2, a software system for research in algebraic geometry.
\newblock Available at \url{http://www.math.uiuc.edu/Macaulay2/}.

\bibitem{GH}
P.~Griffiths and J.~Harris.
\newblock {\em Principles of algebraic geometry}.
\newblock Wiley Classics Library. John Wiley \& Sons, Inc., New York, 1994.
\newblock Reprint of the 1978 original.

\bibitem{Harris79}
J.~Harris.
\newblock Galois groups of enumerative problems.
\newblock {\em Duke Math.~J.}, 46:685--724, 1979.

\bibitem{HRS14}
J.D. Hauenstein, J.I. Rodriguez, and B.~Sturmfels.
\newblock Maximum likelihood for matrices with rank constraints.
\newblock {\em Journal of Algebraic Statistics}, 5:18--38, 2014.

\bibitem{HS10}
J.D. Hauenstein and A.J. Sommese.
\newblock Witness sets of projections.
\newblock {\em Appl. Math. Comput.}, 217(7):3349--3354, 2010.

\bibitem{HSZII}
N.~Hein, F.~Sottile, and I.~Zelenko.
\newblock A congruence modulo four for real {S}chubert calculus with isotropic
  flags, 2016.
\newblock Canadian Mathematical Bulletin, to appear.

\bibitem{HSZI}
N.~Hein, F.~Sottile, and I.~Zelenko.
\newblock A congruence modulo four in real schubert calculus, 2016.
\newblock Journal f\"ur die reine und angewandte Mathematik, to appear.

\bibitem{Hermite}
C.~Hermite.
\newblock Sur les fonctions alg{\'e}briques.
\newblock {\em CR Acad. Sci.(Paris)}, 32:458--461, 1851.

\bibitem{Higman}
D.G. Higman.
\newblock Intersection matrices for finite permutation groups.
\newblock {\em J. Algebra}, 6:22--42, 1967.

\bibitem{HKS05}
S.~Ho{\c{s}}ten, A.~Khetan, and B.~Sturmfels.
\newblock Solving the likelihood equations.
\newblock {\em Found. Comput. Math.}, 5(4):389--407, 2005.

\bibitem{HS14}
J.~Huh and B.~Sturmfels.
\newblock Likelihood geometry.
\newblock In {\em Combinatorial algebraic geometry}, volume 2108 of {\em
  Lecture Notes in Math.}, pages 63--117. Springer, 2014.

\bibitem{J1870}
C.~Jordan.
\newblock {\em Trait\'e des Substitutions}.
\newblock Gauthier-Villars, Paris, 1870.

\bibitem{LS09}
A.~Leykin and F.~Sottile.
\newblock Galois groups of {S}chubert problems via homotopy computation.
\newblock {\em Math. Comp.}, 78(267):1749--1765, 2009.

\bibitem{MSJ}
A.~Mart\'in~del Campo and F.~Sottile.
\newblock Experimentation in the {S}chubert calculus.
\newblock {\tt arXiv.org/1308.3284}, 2016.
\newblock to appear in {\it Schubert Calculus---Osaka 2012}, ed. by H.Naruse,
  T.Ikeda, M.Masuda, and T.Tanisaki, Advanced Studies in Pure Mathematics,
  Mathematical Society of Japan.

\bibitem{MNMH}
D.K. Molzahn, M.~Niemerg, D.~Mehta, and J.D. Hauenstein.
\newblock Investigating the maximum number of real solutions to the power flow
  equations: Analysis of lossless four-bus systems.
\newblock {\tt arXiv/org:1603.05908}, 2016.

\bibitem{RT15}
J.I. Rodriguez and X.~Tang.
\newblock Data-discriminants of likelihood equations.
\newblock In {\em Proceedings of the 2015 ACM on International Symposium on
  Symbolic and Algebraic Computation}, ISSAC '15, pages 307--314, New York, NY,
  USA, 2015. ACM.

\bibitem{RSSS}
J.~Ruffo, Y.~Sivan, E.~Soprunova, and F.~Sottile.
\newblock Experimentation and conjectures in the real {S}chubert calculus for
  flag manifolds.
\newblock {\em Experiment. Math.}, 15(2):199--221, 2006.

\bibitem{OS}
L.L. Scott.
\newblock Representations in characteristic {$p$}.
\newblock In {\em The {S}anta {C}ruz {C}onference on {F}inite {G}roups ({U}niv.
  {C}alifornia, {S}anta {C}ruz, {C}alif., 1979)}, volume~37 of {\em Proc.
  Sympos. Pure Math.}, pages 319--331. Amer. Math. Soc., Providence, R.I.,
  1980.

\bibitem{INAG}
A.J. Sommese, J.~Verschelde, and C.W. Wampler.
\newblock Introduction to numerical algebraic geometry.
\newblock In {\em Solving polynomial equations}, volume~14 of {\em Algorithms
  Comput. Math.}, pages 301--335. Springer, Berlin, 2005.

\bibitem{NAG}
A.J. Sommese and C.W. Wampler.
\newblock Numerical algebraic geometry.
\newblock In {\em The mathematics of numerical analysis ({P}ark {C}ity, {UT},
  1995)}, volume~32 of {\em Lectures in Appl. Math.}, pages 749--763. Amer.
  Math. Soc., Providence, RI, 1996.

\bibitem{SW05}
A.J. Sommese and C.W. Wampler.
\newblock {\em The numerical solution of systems of polynomials}.
\newblock World Scientific Publishing Co. Pte. Ltd., Hackensack, NJ, 2005.

\bibitem{SW08}
A.J. Sommese and C.W. Wampler.
\newblock Exceptional sets and fiber products.
\newblock {\em Found. Comput. Math.}, 8(2):171--196, 2008.

\bibitem{SW_double}
F.~Sottile and J.~White.
\newblock Double transitivity of {G}alois groups in {S}chubert calculus of
  {G}rassmannians.
\newblock {\em Algebr. Geom.}, 2(4):422--445, 2015.

\bibitem{AltBurmesterOrig}
Y.~Tong, D.H. Myszka, and A.P. Murray.
\newblock Four-bar linkage synthesis for a combination of motion and path-point
  generation.
\newblock {\em Proceedings o the ASME Interational Design Engineering Technical
  Conferences}, DETC2013-12969, 2013.

\bibitem{Va}
R.~Vakil.
\newblock Schubert induction.
\newblock {\em Ann. of Math. (2)}, 164(2):489--512, 2006.

\bibitem{Zar}
O.~Zariski.
\newblock A theorem on the {P}oincar\'e group of an algebraic hypersurface.
\newblock {\em Annals of Mathematics}, 38(1):131--141, 1937.

\end{thebibliography}

\end{document}